\documentclass[11pt]{amsart}
\usepackage{graphicx}
\usepackage{xcolor,pinlabel}
\usepackage[all,cmtip]{xy}
\usepackage{amsmath,amssymb,amsfonts,amsthm,mathrsfs}

\newcommand{\co}{\colon\,}
\newcommand{\cR}{\mathcal R}
\newcommand{\SR}{S_{\cR}}
\newcommand{\subm}{\sigma_{\cR}}

\newcommand{\linnum}{\stepcounter{thm}\tag{\thethm}}

\newtheorem{thm}{Theorem}
\newtheorem*{thm*}{Theorem}

\newtheorem{lemma}[thm]{Lemma}

\theoremstyle{definition}
\newtheorem{quest}[thm]{Question}

\begin{document}

\title{Realizing polynomial portraits}

\author[Floyd]{William Floyd}
\address{Department of Mathematics\\ Virginia Tech\\
Blacksburg, VA 24061\\ U.S.A.}
\email{floyd@math.vt.edu}
\urladdr{http://www.math.vt.edu/people/floyd}

\author[Kim]{Daniel Kim}
\address{Department of Mathematics\\ Virginia Tech\\
Blacksburg, VA 24061\\ U.S.A.}
\email{kdani90@vt.edu}

\author[Koch]{Sarah Koch}
\address{Department of Mathematics\\
University of Michigan\\
Ann Arbor, MI 48109\\ U.S.A.}
\email{kochsc@umich.edu}

\author[Parry]{Walter Parry}
\address{Department of Mathematics and Statistics\\
Eastern Michigan University\\ Ypsilanti, MI 48197\\ U.S.A.}
\email{walter.parry@emich.edu}

\author[Saenz]{Edgar Saenz}
\address{Department of Mathematics\\ Virginia Tech\\
Blacksburg, VA 24061\\ U.S.A.}
\email{easaenzm@vt.edu}

\date\today
\thanks{The research of Sarah Koch was supported in part by the NSF}
\maketitle

\begin{abstract}It is well known that the dynamical behavior of a rational 
map $f:\widehat{\mathbb C}\to \widehat{\mathbb C}$ is 
governed by the forward orbits of the critical 
points of $f$. The map $f$ is said to be {\em postcritically finite} 
if every critical point has finite forward orbit, or equivalently, 
if every critical point eventually maps into a periodic cycle of $f$. 
We encode the orbits of the critical points of $f$ with a finite 
directed graph called a {\em ramification portrait}. In this article, 
we study which graphs arise as ramification portraits. We prove that 
every abstract polynomial portrait is realized as the ramification 
portrait of a postcritically finite polynomial, and classify which 
abstract polynomial portraits can only be realized by unobstructed maps.
\end{abstract}

\section{Introduction}

\noindent 
Let $\widehat{\mathbb C}$ denote the Riemann sphere, and let 
$f:\widehat{\mathbb C}\to \widehat{\mathbb C}$ be a rational map of 
degree $d\geq 2$. By the Riemann-Hurwitz formula, $f$ has $2d-2$ 
critical points, counted with multiplicity; the critical set of $f$ 
is denoted $C_f$. The postcritical set of $f$, denoted $P_f$, is 
the smallest forward invariant subset of the Riemann sphere that 
contains the critical values of $f$. If the postcritical set is 
finite, the rational map is said to be {\em postcritically finite}. 
Associated to a postcritically finite rational map is a 
{\em ramification portrait}; 
that is, a finite directed graph $\Gamma_f$ that encodes the action of 
$f$ restricted to $C_f\cup P_f$. As an example, consider the polynomial 
$f:z\mapsto  z^2-2$. The critical set  is $C_f=\{0,\infty\}$, the 
postcritical set is $P_f=\{-2,2,\infty\}$, and the ramification portrait is:
\[
\xymatrix{\infty\ar@(ur,dr)^2}
\quad \xymatrix{0\ar[r]^2 &-2 \ar[r]&2\ar@(ur,dr)}  
\]

\medskip

\noindent In the portrait above, there is an edge from vertex $x$ to
vertex $y$ if and only if $y=f(x)$. This edge is weighted with the
positive integer $\mathrm{deg}_f(x)$, the local degree of $f$ at
$x$. To lighten notation, we record the weight of the edge from $x$ to
$f(x)$ if and only if $\mathrm{deg}_f(x)>1$; that is, if and only if
$x\in C_f$. The portrait above is a {\em polynomial portrait}; that
is, there is a fixed vertex mapping to itself with full degree (the
vertex $\infty$).

In this article, we study which graphs are isomorphic to portraits from
postcritically finite polynomials.  There are immediate necessary
conditions that arise from local degree restrictions, and from
Riemann-Hurwitz restrictions (see Section \ref{sect:preliminaries}).
We prove that in the polynomial setting, these conditions are also
sufficient.  A weighted finite directed graph as above which satisfies
these conditions is called an \emph{abstract polynomial portrait}.
\begin{thm}\label{main:thm:1}
Let $\Gamma$ be an abstract polynomial portrait. Then there exists a 
polynomial $f:\widehat{\mathbb C}\to\widehat{\mathbb C}$ so that 
$\Gamma_f\simeq \Gamma$. 
\end{thm}
\noindent To prove Theorem \ref{main:thm:1}, we  construct an explicit 
{\em topological polynomial} $g:S^2 \to S^2$ so that $\Gamma_g\simeq \Gamma$. 
We build $g$ so that it has no {\em obstructing multicurves}. It then 
follows from Thurston's Topological Characterization of Rational Maps 
that $g$ is {\em combinatorially equivalent} to a polynomial $f$, 
so $\Gamma_f\simeq \Gamma_g$. 

We cannot strengthen Theorem \ref{main:thm:1} by removing the
hypothesis that $\Gamma$ is a polynomial portrait because of the
following two phenomena. The first is dynamical and related to
Thurston's theorem. The second is nondynamical and related to the
Hurwitz problem.

\noindent {\bf Portraits that can only be realized topologically.} 
Consider the following abstract portrait $\Gamma$. 
\[
\xymatrix{a\ar@/^0.7pc/[r]^2 & b \ar@/^0.7pc/[l]} \qquad
\xymatrix{c\ar@/^0.7pc/[r]^2 & d \ar@/^0.7pc/[l]}
\]
\medskip
Suppose there is a rational map $f:\widehat{\mathbb C}\to
\widehat{\mathbb C}$ so that $\Gamma\simeq \Gamma_f$.  Then $f$ is a
quadratic rational map with two periodic cycles of period 2. However,
a quick computation reveals that a quadratic rational map can have at
most one periodic cycle of period $2$, so no such $f$ exists.

Even though no rational map $f$ exists so that $\Gamma_f\simeq \Gamma$, it is 
possible to construct a topological branched cover  $g:S^2\to S^2$ 
so that $\Gamma_g\simeq \Gamma$. For example, after identifying $S^2$ with 
$\widehat{\mathbb C}$, we could take the squaring map $s: z\mapsto z^2$ 
and postcompose with an orientation-preserving homeomorphism 
$h:S^2\to S^2$ so that $h(\infty)=1,h(1)=\infty,h(0)=2$, and $h(4)=0$. 
Then $g:=h\circ s$ is a branched cover with $\Gamma_g\simeq \Gamma$. By 
Thurston's theorem, the map $g$ will necessarily admit an obstructing 
multicurve (see Section \ref{sect:preliminaries}). 

We are aware of a few methods to construct portraits that can only be 
realized topologically that are similar in spirit to the example above. 
It would be interesting to put these examples into a more general context. 

\begin{quest}
Which abstract portraits $\Gamma$ can only be realized topologically? 
\end{quest}

\noindent {\bf Portraits that cannot even be realized topologically.}
The {\em Hurwitz problem} is to characterize which {\em branch data}
arise from branched covering maps $S^2\to S^2$.  See \cite{Bar},
\cite{EKS}, \cite{KZ} and \cite{Thom}.  For example, it is known that
there is no branched cover with the branch data $(2,2),(2,2),(3,1)$.
That is, there is no branched cover $f:S^2\to S^2$ of degree $4$ with
exactly three critical values $\{v_1,v_2,v_3\}\subseteq S^2$, so that
\begin{itemize}
\item $f^{-1}(\{v_1\})$ contains exactly two points, each mapping forward 
with local degree two, 
\item 
$f^{-1}(\{v_2\})$ contains exactly two points, each mapping forward with 
local degree two, and 
\item $f^{-1}(\{v_3\})$ contains exactly two points, one mapping forward 
with local degree 3, and the other mapping forward with local degree 1. 
\end{itemize}
This fact has  dynamical consequences. Indeed, any abstract portrait 
$\Gamma$ with this branch data cannot be the portrait of a branched 
covering map $S^2\to S^2$, and therefore, $\Gamma$ cannot be the portrait 
of a rational map $\widehat{\mathbb C}\to \widehat{\mathbb C}$. 
For example, the following portrait has branch data $(2,2),(2,2),(3,1)$.
 \begin{eqnarray*}
\xymatrix@!R=3pt@!C=7pt{v_1\ar[rd]^2 \\
&v_3\ar[r] &v_4\ar@(ur,dr)\\
v_2\ar[ru]^2 }\qquad 
\xymatrix@!R=3pt@!C=7pt{v_5\ar[rd]^2 \\
&v_7\ar[r] &v_8\ar@/^0.6pc/[r] & v_9 \ar@/^0.6pc/[l]\\
v_6\ar[ru]^2 }\qquad 
\xymatrix@C=13pt{\\
v_{10}\ar[r]^3 &v_{11}\ar@(ur,dr)}
\end{eqnarray*}
\medskip

\noindent While the general Hurwitz problem is unsolved, we note that
all polynomial branch data are realizable (see Proposition 5.2 in
\cite{EKS}). We will not use this fact to construct the branched cover
$g:S^2\to S^2$ in the proof of Theorem \ref{main:thm:1}.

\smallskip

\noindent{\bf Thurston's theorem.} Let $S^2$ denote an oriented topological 
2-sphere, and let $f\co S^2 \to S^2$ be 
an orientation-preserving branched cover of degree $d\ge 2$ so that 
the postcritical set $P_f$ is finite. We call such a map $f$ a 
{\em Thurston map}.  For convenience in stating the theorem, we assume that
the orbifold of $f$ is hyperbolic\footnote{This condition essentially 
excludes power maps $z\mapsto z^n$, Chebyshev maps, and Latt\`es maps; see
\cite{DH}.}.
Two Thurston maps $f:(S^2,P_f)\to (S^2,P_f)$ 
and $g:(S^2,P_g)\to (S^2,P_g)$ are {\em combinatorially equivalent} 
provided that there are orientation-preserving homeomorphisms 
$\phi_0:(S^2,P_f)\to (S^2,P_g)$ and $\phi_1:(S^2,P_f)\to (S^2,P_g)$ so that 
\begin{itemize}
\item $\phi_0\circ f = g\circ \phi_1$, and 
\item the homeomorphisms $\phi_0$ and $\phi_1$ are isotopic relative to $P_f$. 
\end{itemize}

 In the 1980s, William Thurston proved that every Thurston map $f$ is 
combinatorially equivalent to a rational map, or it is {\em obstructed}. 
In the latter case, $f$ admits an invariant curve system called an 
{\em obstructing multicurve}.

A {\em multicurve} $\Delta$ is a finite collection of simple disjoint 
curves in $S^2\setminus P_f$, no two of which are homotopic. All 
components $\delta\in \Delta$ are also required to be {\em essential}  
($\delta$ does not bound a disk), and nonperipheral ($\delta$ does not 
bound a disk with exactly one puncture). The multicurve $\Delta$ is 
said to be {\em invariant} for $f$ provided that for all 
$\delta\in\Delta$, every component of $f^{-1}(\{\delta\})$ is either
 \begin{itemize}
 \item homotopic to some $\delta'\in\Delta$ in $S^2\setminus P_f$, or 
 \item `erased'; that is, it is peripheral or inessential. 
 \end{itemize}

Given an invariant multicurve $\Delta$ for $f$, Thurston defined an 
associated linear transformation ${\mathbb R}^\Delta\to{\mathbb R}^\Delta$ 
that encodes how different components of $f^{-1}(\Delta)$ map to $\Delta$. 
The matrix for this transformation has non-negative real entries, so there 
is a leading eigenvalue $\lambda$ which is real and non-negative. The 
multicurve $\Delta$ is an {\em obstruction} provided that $\lambda\geq 1$. 
If the Thurston map $f$ admits an obstruction, $f$ is said to be 
{\em obstructed}. If not, $f$ is said to be {\em unobstructed}. 

\begin{thm*}[Thurston's Topological Characterization of Rational
Maps, \cite{DH}]Let $f:(S^2,P_f)\to
(S^2,P_f)$ be a Thurston map, and suppose that $f$ has a hyperbolic
orbifold.  Then $f$ is combinatorially equivalent to a rational map
$F$ if and only if $f$ is unobstructed. In this case, $F$ is unique up
to conjugation by M\"obius transformations.
\end{thm*}

\smallskip
\noindent {\bf Levy cycles.} For a given Thurston map $f:(S^2,P_f)\to
(S^2,P_f)$, verifying the criterion in Thurston's theorem is difficult
as it involves an infinite search in general. In this article, we will
work with Thurston maps that are {\em topological polynomials}; that
is, there is some $\omega\in S^2$ that is a fully ramified fixed point
of $f$. More is known about Thurston's criterion in the case of
topological polynomials.

A {\em Levy cycle} for the Thurston map $f:(S^2,P_f)\to (S^2,P_f)$ is 
a circularly ordered collection of simple closed curves 
$\{\delta_0,\ldots, \delta_{n-1},\delta_{n}=\delta_0\}$ on 
$S^2\setminus P_f$ such that 
\begin{itemize}
\item no two curves are homotopic relative to $P_f$,  
\item the curves are pairwise disjoint, 
\item each curve is essential and nonperipheral, and 
\item for all $1\leq i\leq n$, at least one component of $f^{-1}(\delta_i)$ 
is homotopic to $\delta_{i-1}$ and maps to $\delta_i$ by degree $1$. 
\end{itemize}

Silvio Levy proved the following results in his thesis, \cite{Levy}. 

\begin{thm*}[Levy] Let $f:(S^2,P_f)\to (S^2,P_f)$ be a Thurston map
that is a topological polynomial. Then $f$ is obstructed if and only
if $f$ admits a Levy cycle.
\end{thm*}

\begin{thm*}[Levy] Let $\Gamma$ be an abstract polynomial portrait
such that every critical vertex is periodic.  Then every Thurston map
realizing $\Gamma$ is unobstructed.
\end{thm*}

The proof of the latter can be strengthened to give the following result.  
See, for example, Hubbard \cite[Theorem 10.3.9]{Hub}.

\begin{thm*}[Levy-Berstein] Suppose $\Gamma$ is an abstract polynomial
portrait such that each cycle contains a critical vertex.  Then every
Thurston map realizing $\Gamma$ is unobstructed.
\end{thm*}

We will use Levy's first theorem in an essential way in our proof 
of Theorem \ref{main:thm:1}. Indeed, given an abstract portrait $\Gamma$, 
we will construct a topological polynomial $g:S^2\to S^2$ so that 
$\Gamma_g\simeq \Gamma$, and so that $g$ cannot possibly admit a Levy cycle. 
Theorem \ref{main:thm:1} immediately follows.  

 In \cite[Theorem 1.1]{Kel} Kelsey uses self-similar groups to give a
partial converse to the Levy-Berstein theorem.  In the discussion that
follows, an {\em attractor} of an abstract portrait is a cycle that
contains a critical vertex, and a {\em non-attractor} is a cycle that
does not contain a critical vertex.

\begin{thm*}[Kelsey]
Suppose $\Gamma$ is an abstract polynomial portrait, and that $\Gamma$ 
satisfies at least one of the following properties:
\begin{itemize}
\item[(1)] $\Gamma$ contains a cycle\footnote{which is necessarily a 
non-attractor} of length at least two that does not contain any 
critical values;
\item[(2)] $\Gamma$ contains at least two cycles\footnote{which are 
necessarily non-attractors} that do not contain any critical values;
\item[(3)] $\Gamma$ contains at least two non-attractor cycles that have 
length at least two;
\item[(4)] $\Gamma$ contains at least four non-attractor cycles.
\end{itemize}
Then there is an obstructed Thurston map that realizes $\Gamma$.
\end{thm*}

In Theorem \ref{thm:unobstructed}, we show that certain abstract polynomial
portraits have only unobstructed representatives, and in Theorem
\ref{thm:obstructed}, we show that certain abstract polynomial portraits have
obstructed representatives.  We need a definition to state
Theorems~\ref{thm:unobstructed} and \ref{thm:obstructed}.  Let
$\Gamma$ be an abstract polynomial portrait, and let $v$ be a vertex
of $\Gamma$.  Then $v$ is the source vertex of exactly one edge of
$\Gamma$.  We let $\tau(v)$ denote the target vertex of this edge.

\begin{thm}\label{thm:unobstructed} Suppose $\Gamma$ is an abstract
polynomial portrait that has at least four postcritical vertices and
satisfies one of the following properties.
\begin{itemize}
\item[(i)] $\Gamma$ has a single non-attractor cycle, and it has length one.
\item[(ii)] Every finite postcritical vertex of $\Gamma$ is in a single 
non-attractor cycle, this cycle has length $p^k$ for some prime number $p$ 
and some positive integer $k$, and the finite postcritical vertices can be 
enumerated as $\{v_i: 0\leq i < p^k\}$ such that $\tau(v_i) = v_{i+1}$ 
(mod $p^k$) for every $i\in \{0,\dots,p^k -1\}$, and
if $v_j$ is a critical value then $j$ is a multiple of $p^{k-1}$.
\end{itemize}
Then every Thurston map with portrait isomorphic to $\Gamma$ is
unobsructed.
\end{thm}

The hypothesis that there are at least four postcritical vertices is
not restrictive, since by Thurston's characterization theorem a
Thurston map with fewer than four postcritical points is unobstructed.
The proof is along the lines of the argument for the Levy-Berstein
Theorem.  If the abstract portrait can be realized by an obstructed
Thurston map, then by Levy \cite{Levy} there must be a Levy cycle.
This implies that, in the teminology of Hubbard \cite{Hub}, there must
be a {\em degenerate Levy cycle}.  One then shows that this is
impossible if the portrait satisfies (i) or (ii).  The proof is given
in Section \ref{sec:unobstructed}.  Part (if not all) of case (i) of
Theorem \ref{thm:unobstructed} was previously known.  The case of a
single non-attractor cycle of length one and no other finite cycles
was observed by Kelsey \cite[p. 52]{Kel}.

\begin{thm}\label{thm:obstructed} Suppose $\Gamma$ is an abstract
polynomial portrait that has at least four postcritical vertices and
satisfies one of the following properties.
\begin{itemize}
\item[(i)] Every finite postcritical vertex of $\Gamma$ is in a single 
non-attractor cycle, this cycle has length $p^k$ for some prime number 
$p$ and some positive integer $k$,
the finite vertices can be enumerated as $\{v_i: 0\le i < p^k\}$ such
that $v_0$ is a critical value, $\tau(v_i) = v_{i+1}$ mod $p^k$ for every
$i\in \{0,\dots,p^k -1\}$, and there is a critical value $v_j$ such that
$j$ is not a multiple of $p^{k-1}$.
\item[(ii)] Every finite postcritical vertex of $\Gamma$ is in a single 
non-attractor cycle of length at least two, 
and this cycle does not have prime-power length.
\item[(iii)]$ \Gamma$ contains a non-attractor cycle of length at least 
two that does not
contain all of the finite critical values.
\item[(iv)] $\Gamma$ has at least two non-attractor cycles of length one.
\end{itemize}
Then there exists an obstructed Thurston map whose   portrait is 
isomorphic to $\Gamma$.\end{thm}

The proof of Theorem \ref{thm:obstructed} is constructive and relies
on a combinatorial lemma, Lemma \ref{lem:cmbllemma}.  Given an
abstract polynomial portrait $\Gamma$ that satisfies any of conditions
(i)-(iv) of the theorem, we describe a construction of an obstructed
Thurston map with portrait isomorphic to $\Gamma$. We introduce {\em
rose maps} and prove the lemma in Section \ref{sec:rose}. We then
prove the theorem in Section \ref{sec:obstructed}.

Combining Theorem \ref{thm:unobstructed}, Theorem
\ref{thm:obstructed}, and the Levy-Berstein Theorem, we classify the
abstract polynomial portraits that are completely unobstructed.  We
summarize this result in the following theorem, which we prove in
Section \ref{sec:classification}.

\begin{thm}\label{thm:classification} Suppose $\Gamma$ is an abstract
polynomial portrait. Then every Thurston map with portrait isomorphic
to $\Gamma$ is unobstructed if and only if $\Gamma$ satisfies at least
one of the following conditions.
\begin{itemize}
\item[(i)]
$\Gamma$ has at most three postcritical vertices.
\item[(ii)]
Every cycle of $\Gamma$ is an attractor.
\item[(iii)]
$\Gamma$ has a single non-attractor cycle, and it has length one.
\item[(iv)]
Every finite postcritical vertex of $\Gamma$ is in a single non-attractor 
cycle, this cycle has length $p^k$ for some prime number $p$ and some 
positive integer $k$, the finite postcritical vertices can be enumerated 
as 
\[
\{v_i: 0\le i < p^k\} \text{ such
that }\tau(v_i) = v_{i+1}\;\mathrm{mod}\; p^k \text{ for every }i\in \{0,\dots,p^k -1\},
\]
and
if $v_j$ is a critical value, then $j$ is a multiple of $p^{k-1}$.
\end{itemize}
\end{thm}

\noindent {\bf Notes and references.} In addition to the references we have already highlighted, we would like to mention two somewhat related works: in \cite{DKM}, the authors prove that given a finite set $X$ with $|X|\geq 2$, a map $F:X\to X$, and prescribed multiplicities at points of $X$, there is some rational map $f:\widehat{\mathbb C}\to\widehat{\mathbb C}$ so that the restriction of $f$ to its postcritical set coincides with the map $F$ in such a way that the multiplicity of each point $x\in X$ agrees with the local degree of $f$ at the corresponding postcritical point. We note that the map $F:X\to X$, together with the multiplicities at points of $X$,  is analogous to our abstract ramification portrait $\Gamma$; however, our abstract portraits satisfy a Riemann-Hurwitz condition, so they have a natural degree, $d\geq 2$. In our article, we are interested in the question of whether there is a rational map of degree $d$ that realizes a given abstract portrait. In \cite{DKM}, the degree of the rational maps constructed can be arbitrarily large. 

In \cite{DS}, the authors begin with a similar sort of abstract
combinatorial object called a portrait, and they construct a moduli
space of endomorphisms $\mathbb P^N\to \mathbb P^N$ (using GIT) that
consists of points realizing the given combinatorial data.  While we
prove existence of maps with a given abstract portrait, in \cite{DS}
it is not proven that the moduli spaces are nonempty.

\section{Preliminaries}\label{sect:preliminaries}

\noindent {\bf Portraits associated to Thurston maps.} Let
$f:(S^2,P_f)\to (S^2,P_f)$ be a Thurston map of degree $d$. The
\emph{ramification portrait} of $f$ is the weighted directed graph
$\Gamma$ such that the vertex set $V(\Gamma)$ is the union of the set
$C_f$ of critical points and the set $P_f$ of postcritical points, and
for each vertex $v$ there is an edge from $v$ to $f(v)$ with weight
the local degree $\deg_f(v)$ of $f$ at $v$.  By the Riemann-Hurwitz
formula,

\[
\sum_{v\in C_f} (\deg_f(v)-1) = 2d-2.
\]
Since $f$ has degree $d$, at each vertex $v$ the sum of the weights of
the incoming edges is at most $d$.  Note that $f$ is a topological
polynomial if and only if there is a vertex $v$ such that $f(v) = v$
and $\deg_f(v)=d$.

\smallskip

\noindent {\bf Abstract portraits.} Suppose $\Gamma$ is a finite
weighted directed graph (with the weights positive integers) such that
each vertex of $\Gamma$ is the source of exactly one edge.  Let
$\tau\co V(\Gamma)\to V(\Gamma)$ be the function which takes a vertex
$v$ to the target of the edge with source $v$.  We call the weight of
the edge from $v$ to $\tau(v)$ the \emph{degree} of $\tau$ at $v$ and
denote it by $\deg(v)$.  A vertex $v$ is \emph{critical} if $\deg(v) >
1$, and is \emph{postcritical} if there are a critical vertex $w$ and
a positive integer $k$ such that $\tau^{\circ k}(w) = v$. If $v$ is a
critical vertex, then $\tau(v)$ is called a \emph{critical value}.  We
denote the set of critical vertices by $C_{\Gamma}$, and we denote the
set of postcritical vertices by $P_{\Gamma}$. We say that $\Gamma$ is
an \emph{abstract portrait} if it satisfies the following:
\begin{itemize}
\item every vertex of $\Gamma$ is either critical or postcritical, 
\item there is an integer $d\ge 2$ such that 
$\sum_{v\in C_{\Gamma}} (\deg(v)-1) = 2d-2$, and 
\item for each vertex $v$ the sum of
the weights of the edges with target $v$ is at most $d$. 
\end{itemize}
We call $d$ the
\emph{degree} of the abstract portrait.
We say that an abstract portrait $\Gamma$ is \emph{realized} by a
Thurston map $f$ if $\Gamma$ is isomorphic to the portrait of $f$ 
(as weighted directed graphs).  An abstract portrait $\Gamma$ is
\emph{realizable} if it is realized by some Thurston map.

An abstract portrait of degree $d$ is an \emph{abstract polynomial
portrait} if there is a vertex $v$ such that $\tau(v)=v$ and
$\deg(v)=d$.  In this case we choose such a vertex and call it
$\infty$; the other vertices are called \emph{finite}. We call a cycle
(of the action of $\tau$ on $V(\Gamma)$) \emph{finite} if all of its
vertices are finite; that is, a cycle is finite if it does not consist
of the singleton $\infty$.

\smallskip

\noindent\textbf{Finite subdivision rules.} We define finite
subdivision rules in the present context of Thurston maps.  A finite
subdivision rule $\cR$ consists of the structure of a finite
CW complex $S_\cR$ on the 2-sphere (called the model subdivision complex), a
subdivision $\cR(S_\cR)$ of $S_\cR$ and a continuous cellular map
$\sigma_\cR\co \cR(S_\cR)\to S_\cR$ (called the \emph{subdivision
map}) whose restriction to each open cell is a homeomorphism onto an
open cell.  Furthermore, for each closed 2-cell $\widetilde{t}$ of
$S_\cR$ there are (i) a cell structure $t$ (called the \emph{tile
type} of $\widetilde{t}$) on the 2-disk $D^2$ such that the 1-skeleton
of $t$ is $\partial D^2$ and (ii) a continuous surjection $\psi_t\co
t\to \widetilde{t}$ (called the \emph{characteristic map} of
$\widetilde{t}$) whose restriction to each open cell is a
homeomorphism onto an open cell.

The map $\sigma_\cR$ is a Thurston map if it has degree at least 2.
Conversely, a Thurston map $f$ is the subdivision map of a finite
subdivision rule if and only if there exists a connected finite
$f$-invariant graph $G$ which contains the postcritical set of $f$.
Such a graph $G$ serves as the 1-skeleton of a model subdivision
complex.

\section{Realizing a portrait by an unobstructed map}\label{sec:poly}

In this section we prove Theorem \ref{main:thm:1}.  We begin with an example 
to illustrate the construction. Consider the abstract portrait 
$\Gamma$ that is shown below.

\[
\xymatrix{
\ast \ar[r]^2 & a \ar[r] & b \ar@/_/[r] & c \ar@/_/[l] & e \ar@/_/[r]_2 & f
\ar@/_/[l] & \infty \ar@(ur,dr)^4 \\
& \ast \ar[r]^2 & d \ar[ur]
}
\]

The proof defines an ordering of the finite postcritical vertices of
$\Gamma$.  In this case we use the ordering given by $a < b < c < d <
e < f$.  Following the terminology that will be defined in the proof,
the ordered sets $(a,b,c)$ and $(d)$ are called type-$1$ chains and
the ordered set $(e,f)$ is called a type-$2$ chain. 
(The first element of a type-$1$ chain is the image of a critical vertex
that is not postcritical, and the first element of a type-$2$ chain is
a periodic critical vertex.)  The model
subdivision complex $\SR$ is shown in Figure~\ref{fig:stickers} as a
stereographic projection of $S^2$ to the plane.  The 1-skeleton will
always be a star graph with central vertex $\infty$. The vertices of
$\Gamma$ are identified with the vertices of $\SR$.  The ordering of
the finite postcritical vertices chosen above determines the
counterclockwise ordering of the labels of the vertices in
Figure~\ref{fig:stickers}. The tile type $t$ is shown in
Figure~\ref{fig:tiletype}; $\SR$ is the image of $t$ under the
characteristic map $\psi\co t\to \SR$.  The \emph{label} of a vertex
$v$ of $t$ is $\psi(v)$; if $\psi(v)\ne\infty$ then $v$ is called a
\emph{finite} vertex.

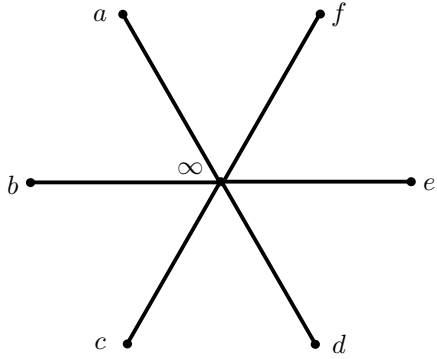
\begin{figure}[!ht]
\labellist
\small\hair 2pt
\pinlabel $a$ at 28 126
\pinlabel $b$ at -5 62
\pinlabel $c$ at 28 2
\pinlabel $d$ at 118 2
\pinlabel $e$ at 152 62
\pinlabel $f$ at 118 126
\pinlabel $\infty$ at 62 68
\endlabellist
\centering
\includegraphics{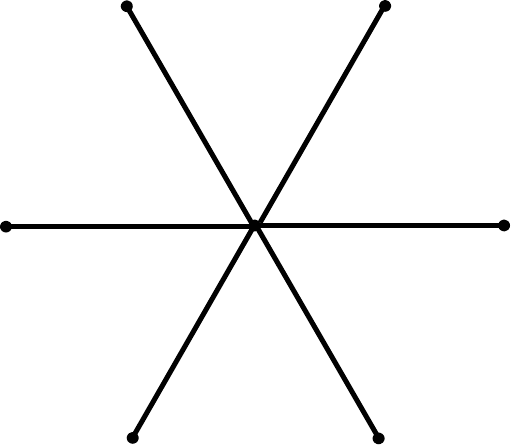}
\caption{The model subdivision complex $\SR$}
\label{fig:stickers}
\end{figure}

\begin{figure}[!ht]
\labellist
\small\hair 2pt
\pinlabel $a$ at 76 130
\pinlabel $b$ at 133 95
\pinlabel $c$ at 133 30
\pinlabel $d$ at 76 -5
\pinlabel $e$ at  14 30
\pinlabel $f$ at 14 95
\pinlabel $\infty$ at 112 130
\pinlabel $\infty$ at 153 63
\pinlabel $\infty$ at 112 -4 
\pinlabel $\infty$ at 33 -4
\pinlabel $\infty$ at -6 63
\pinlabel $\infty$ at 32 130
\endlabellist
\centering
\includegraphics{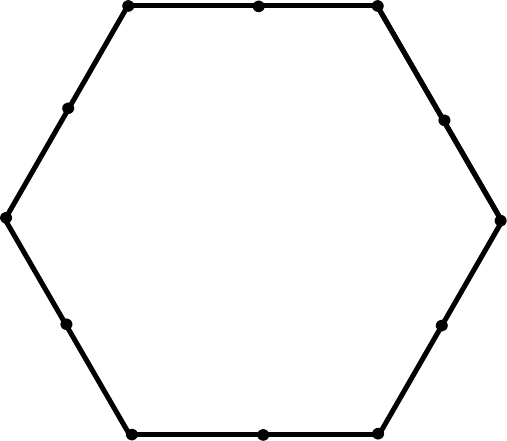}
\caption{The tile type $t$, which $\psi$ maps to $\SR$ by identifying
edges in pairwise fashion} \label{fig:tiletype}
\end{figure}

We will give a combinatorial description of the subdivision
$\cR(\SR)$.  We first add edges to $\SR$ that will ensure that the
subdivision map cannot have any Levy cycles (stage $1$), and then add
more edges to get $\cR(\SR)$ (stages 2 and 3).
Figure~\ref{fig:stage1} shows the construction after the first stage
from the point of view of the tile type $t$.  No further changes are
made in the second stage since there are already the correct number of
subtiles.  The \emph{label} of each vertex is drawn outside $t$.
Every vertex of the subdivision whose label is not $\infty$ is a
\emph{finite} vertex. Every finite vertex $v$ has an \emph{image
label}, which is $\subm(\psi(v))$. It is drawn inside $t$.  (Of
course, this is abuse of notation, since we haven't finished the
construction yet and hence haven't defined the subdivision map $\subm$
yet.)

To complete the construction (stage 3) we add stickers as needed in
each subtile so that each subtile is a $12$-gon, every other vertex is
the original vertex labeled $\infty$, and the image labels of its
finite vertices are in the proper cyclic order.  (A sticker is an edge
with a vertex of valence one, resembling a stick pin with a spherical
head.) It is straightforward to define the subdivision map $\subm$ so
that its restriction to each open cell is a homeomorphism to an open
cell and it takes each finite vertex to its image label.
Figure~\ref{fig:stage3} shows the subdivision of the tile type $t$,
and Figure~\ref{fig:firstsub} shows the subdivision $\cR(\SR)$.

If $\gamma$ is a simple closed curve in $S^2 \setminus P$, let
$D_{\gamma}$ be the component of $S^2\setminus \gamma$ that does not
contain $\infty$.  If $\gamma$ is an element of a Levy cycle (or, more
generally, of a multicurve), then $D_{\gamma}$ must contain at least
two postcritical points.  The five new edges in
Figure~\ref{fig:stage1} ensure that if we extend the subtiling so that
it combinatorially describes a finite subdivision rule, then the
subdivision map cannot have a Levy cycle.  This can be proven as
follows.  In the model subdivision complex, the new arc whose
barycenter has label $d$ bounds a closed disk $D$ such that
$\textrm{int}(D) \cap P_f = \{c\}$ and its boundary contains $\infty$.
It follows from Lemma \ref{lem:Levyedge} that for any positive integer
$n$, each element of a Levy cycle can be isotoped rel the postcritical
set to be disjoint from all new edges of the $n^{\text{th}}$
subdivision $\cR^{n}(\SR)$.  Since the interior of the disk $D$
contains the single postcritical point $c$ and its boundary contains
$\infty$, the vertex $c$ cannot be in the open disk $D_{\gamma}$ for a
Levy curve $\gamma$. In the next two subdivisions there will be new
edges enclosing the stickers with vertices $d$, $b$ and $a$, so none
of these vertices could be in the open disk $D_{\gamma}$ for a Levy
curve $\gamma$.  There is a new edge joining the vertex $e$ to an
$\infty$-vertex, so the vertex $e$ cannot be in the open disk
$D_{\gamma}$ for a Levy curve $\gamma$.  In the next subdivision there
will be a new edge joining the vertex labeled $f$ to an
$\infty$-vertex, so that vertex cannot be in the open disk
$D_{\gamma}$ for a Levy curve $\gamma$.  Hence no finite vertex can be
in a Levy disk, so there are no Levy cycles and hence the subdivision
map is equivalent to a rational map.  This concludes our example.

\begin{figure}[!ht]
\labellist
\small\hair 2pt
\pinlabel $a$ at 76 131
\pinlabel $b$ at 133 95
\pinlabel $c$ at 133 30
\pinlabel $d$ at 76 -5
\pinlabel $e$ at  14 30
\pinlabel $f$ at 14 95
\pinlabel $\infty$ at 112 130
\pinlabel $\infty$ at 153 63
\pinlabel $\infty$ at 112 -4
\pinlabel $\infty$ at 33 -4
\pinlabel $\infty$ at -6 63
\pinlabel $\infty$ at 32 130
\pinlabel $\textcolor{darkgray}{b}$ at 76 120
\pinlabel $\textcolor{darkgray}{c}$ at 119 90
\pinlabel $\textcolor{darkgray}{b}$ at 122 34
\pinlabel $\textcolor{darkgray}{c}$ at 75 6
\pinlabel $\textcolor{darkgray}{f}$ at 28 31
\pinlabel $\textcolor{darkgray}{f}$ at 21 40 
\pinlabel $\textcolor{darkgray}{e}$ at 24 94
\pinlabel $\textcolor{darkgray}{d}$ at 101 48
\pinlabel $\textcolor{darkgray}{a}$ at 47 83
\endlabellist
\centering
\includegraphics{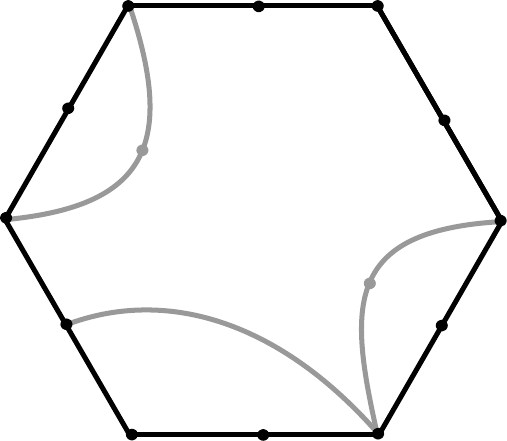}
\caption{The construction after stage $1$}
\label{fig:stage1}
\end{figure}
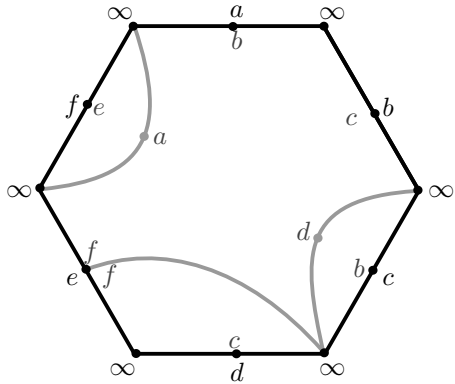

\begin{figure}[!ht]
\labellist
\small\hair 2pt
\pinlabel $a$ at 76 131
\pinlabel $b$ at 133 95
\pinlabel $c$ at 133 30
\pinlabel $d$ at 76 -5
\pinlabel $e$ at  14 30
\pinlabel $f$ at 14 95
\pinlabel $\infty$ at 112 130
\pinlabel $\infty$ at 153 63
\pinlabel $\infty$ at 112 -4
\pinlabel $\infty$ at 33 -4
\pinlabel $\infty$ at -6 63
\pinlabel $\infty$ at 32 130
\pinlabel $\textcolor{darkgray}{b}$ at 76 120
\pinlabel $\textcolor{darkgray}{c}$ at 119 90
\pinlabel $\textcolor{darkgray}{b}$ at 122 34
\pinlabel $\textcolor{darkgray}{c}$ at 73 6
\pinlabel $\textcolor{darkgray}{f}$ at 28 31
\pinlabel $\textcolor{darkgray}{f}$ at 21 41
\pinlabel $\textcolor{darkgray}{e}$ at 24 95
\pinlabel $\textcolor{darkgray}{d}$ at 101 48
\pinlabel $\textcolor{darkgray}{a}$ at 47 83
\pinlabel $\textcolor{darkgray}{f}$ at 36 93
\pinlabel $\textcolor{darkgray}{b}$ at 26 76
\pinlabel $\textcolor{darkgray}{c}$ at 29 81
\pinlabel $\textcolor{darkgray}{d}$ at 25 88
\pinlabel $\textcolor{darkgray}{e}$ at 86 45
\pinlabel $\textcolor{darkgray}{c}$ at 116 29
\pinlabel $\textcolor{darkgray}{e}$ at 114 50
\pinlabel $\textcolor{darkgray}{f}$ at 118 43
\pinlabel $\textcolor{darkgray}{a}$ at 127 40
\pinlabel $\textcolor{darkgray}{d}$ at 57 18
\pinlabel $\textcolor{darkgray}{e}$ at 41 26
\pinlabel $\textcolor{darkgray}{a}$ at 77 17
\pinlabel $\textcolor{darkgray}{b}$ at 77 10
\endlabellist
\centering
\includegraphics{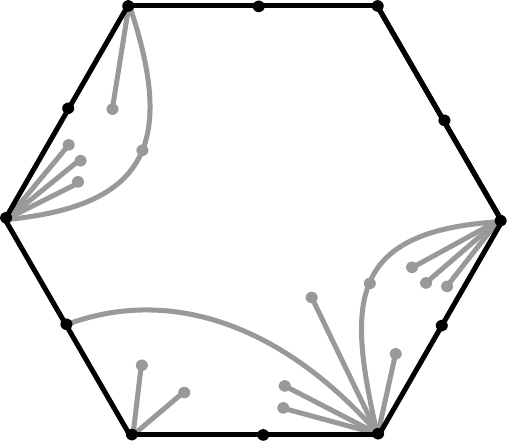}
\caption{The subdivision of the tile type after stage $3$}
\label{fig:stage3}
\end{figure}

\begin{figure}[!ht]
\labellist
\small\hair 2pt
\pinlabel $a$ at 30 162
\pinlabel $\textcolor{darkgray}{b}$ at 36 168
\pinlabel $b$ at 0 103
\pinlabel $\textcolor{darkgray}{c}$ at 0 92
\pinlabel $c$ at 34 32
\pinlabel $\textcolor{darkgray}{b}$ at 42 32
\pinlabel $f$ at 106 166
\pinlabel $\textcolor{darkgray}{e}$ at 115 166
\pinlabel $e$ at 140 103
\pinlabel $\textcolor{darkgray}{f}$ at 140 89
\pinlabel $d$ at 112 42
\pinlabel $\textcolor{darkgray}{c}$ at 112 32
\pinlabel $\infty$ at 63 102
\pinlabel $\textcolor{darkgray}{a}$ at 128 177
\pinlabel $\textcolor{darkgray}{d}$ at 17 11
\pinlabel $\textcolor{darkgray}{e}$ at 30 63
\pinlabel $\textcolor{darkgray}{f}$ at 30 55
\pinlabel $\textcolor{darkgray}{a}$ at 31 45
\pinlabel $\textcolor{darkgray}{c}$ at 52 42
\pinlabel $\textcolor{darkgray}{e}$ at 73 47
\pinlabel $\textcolor{darkgray}{a}$ at 91 42
\pinlabel $\textcolor{darkgray}{b}$ at 100 37
\pinlabel $\textcolor{darkgray}{d}$ at 108 66
\pinlabel $\textcolor{darkgray}{e}$ at 119 79
\pinlabel $\textcolor{darkgray}{f}$ at 92 154
\pinlabel $\textcolor{darkgray}{b}$ at 113 135
\pinlabel $\textcolor{darkgray}{c}$ at 113 142
\pinlabel $\textcolor{darkgray}{d}$ at 112 154
\endlabellist
\centering
\includegraphics{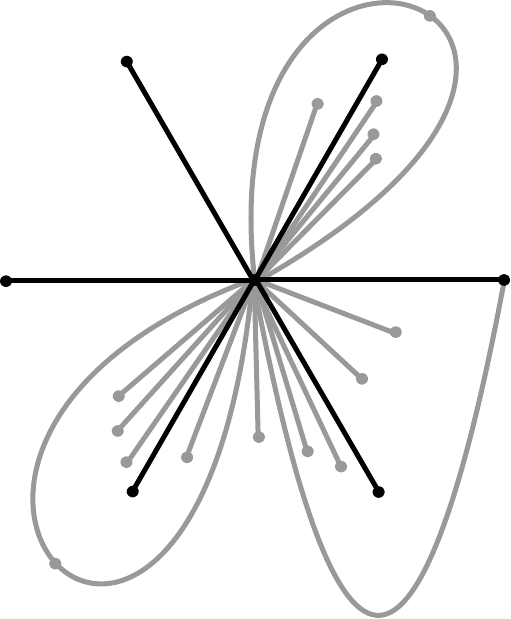}
\caption{The subdivision of the model subdivision complex}
\label{fig:firstsub}
\end{figure}

\bigskip

We call an edge of a subdivision $\cR^n(\SR)$ (of a finite 
subdivision rule $\cR$) a \emph{new edge} if it is not contained in an edge
of $\SR$.  The following lemma plays a crucial role in the proof of
Theorem \ref{main:thm:1}.

\begin{lemma}\label{lem:Levyedge} Suppose $f$ is a Thurston map which
is also the subdivision map $\subm$ of a finite subdivision rule
$\cR$.  Suppose $\{\delta_0,\dots,\delta_{k-1}, \delta_{k}=\delta_0\}$
is a Levy cycle for $f$ and let $n$ be a positive integer.  Then for
each $i\in \{1,\dots,k\}$, $\delta_i$ can be isotoped rel $P_f$ so
that it is disjoint from each new edge of the subdivision
$\cR^n(\SR)$.
\end{lemma}
\begin{proof} We first assume $n=1$.  Let $E$ be the $1$-skeleton of
$\SR$ and let $E_1$ be the $1$-skeleton of $\cR(\SR)$.  Each
$\delta_i$ can be isotoped so that $\delta_i\cap E_1$ is finite.  For
each $i\in \{0,\dots,k\}$, let $a_i$ be the minimum of $\#(\delta \cap
E)$, where $\delta$ is a curve that is isotopic rel $P_f$ to
$\delta_i$, and let $b_i$ be the minimum of $\#(\delta\cap E_1)$,
where $\delta$ is a curve that is isotopic rel $P_f$ to $\delta_i$.
Let $i\in\{1,\dots,k\}$, and let $\delta$ be a curve that is isotopic
rel $P_f$ to $\delta_i$ such that $a_i = \#(\delta \cap E)$.  Let
$\gamma$ be a component of $f^{-1}(\delta)$ which maps to $\delta$ by
degree $1$ and is isotopic rel $P_f$ to $\delta_{i-1}$.  Then $a_{i-1}
\le b_{i-1} \le \#(\gamma \cap E_1) = a_i$.  Since this is true for
every $i$ and $i$ varies cyclically, each of these inequalities is an
equality.  So $a_{i-1} = a_i$ and $a_{i-1} = b_{i-1}$. This implies
that $\gamma$ doesn't intersect $E_1 \setminus E$.  This establishes
the result for $n=1$.

Now suppose that $n>1$. Let $p$ be a positive integer with $p\ge n$ and 
$p\equiv 1 \mod k$.  Then  $\{\delta_0,\dots,\delta_{k-1},
\delta_{k}=\delta_0\}$ is a Levy cycle for $f^{\circ p}$.  
By the previous paragraph applied to
$f^{\circ p}$, each $\delta_i$ can be isotoped so that it does not
intersect any new edge of $\cR^p(\SR)$.  Since each new edge of $\cR^n(\SR)$
is a union of new edges of $\cR^p(\SR)$, then each $\delta_i$ can be
isotoped so that it is disjoint from each new edge of $\cR^n(\SR)$.
\end{proof}

\begin{proof}[Proof of Theorem \ref{main:thm:1}] Let $\Gamma$ be an
abstract polynomial portrait.  Let $C' = C_{\Gamma}\setminus
\{\infty\}$ (the set of finite critical vertices), and let $P' =
P_{\Gamma} \setminus \{\infty\}$ (the set of finite postcritical
vertices).  Let $V_{\Gamma} = \{\tau(x): x\in C_{\Gamma}\}$ (the set of
critical values) and let $V' = V_{\Gamma} \setminus \{\infty\}$ (the
set of finite critical values).  Let $A = \{v\in V': v = \tau(c)\
\textrm{for some}\ c\in C_{\Gamma}\setminus P_{\Gamma}\}$.  For each
$v\in A$, we choose an element $c_v \in C_{\Gamma}\setminus
P_{\Gamma}$ with $\tau(c_v) = v$.

Let $n$ be the cardinality of $P'$. A key step is to appropriately
order the elements of $P'$ by naming them $a_1,\dots,a_n$.  To do
this, we partition $P'$ into \emph{chains}.  We define the chains
recursively. We will put postcritical vertices that have already been
placed in chains in a set $\widetilde{A}$. To begin the construction,
let $\widetilde{A} = \emptyset$ and let $i=1$.
\medskip

\noindent\textbf{The ordering.} Suppose for the recursive step that
$A\not\subset \widetilde{A}$, $i\in \{1,\dots,n\}$, and that we have already
defined $a_j$ for $j\in \{1,\dots,i-1\}$.  
If there is a vertex in $P'\setminus
\widetilde{A}$ that is not periodic under $\tau$, then we can choose
an element $v\in A\setminus \widetilde{A}$ such that $v$ is not the
image under $\tau$ of a postcritical vertex.  If every vertex in
$P'\setminus \widetilde{A}$ is periodic under $\tau$, choose $v\in
A\setminus \widetilde{A}$.  In each case, let $a_i = v$ and add $v$ to
$\widetilde{A}$.  If $\tau(v) \notin \widetilde{A}$, we let $a_{i+1} =
\tau(v)$ and add $\tau(v)$ to $\widetilde{A}$.  We continue until we
reach an index $j$ such that $\tau(a_j)$ is in $\widetilde{A}$.  At
this point we stop this iteration of the recursion. We define the
ordered set $(a_i,\dots,a_{j})$ to be the \emph{chain} of each of its
elements. We call it a \emph{type-$1$} chain.  It begins with an
element of $A$. The \emph{first element} of the chain is $a_i$, and
the \emph{last} element of the chain is $a_{j}$. The \emph{length} of
the chain is $j+1-i$. After redefining $i$ to be $j+1$, we continue
this recursive step as long as possible.

Once we can no longer continue this recursion, the elements of $P'$
which remain are exactly the elements of finite (attractor) cycles
which are connected components of $\Gamma$.  To start the next
recursion, we choose a critical vertex $v$ in a remaining attractor
cycle and let $a_i = v$.  Let $k$ be the number of elements in the
attractor cycle. For $1\le j < k$, let $a_{i+j} = \tau^{\circ
j}(a_i)$. As before $a_i$ is the \emph{first} element of the chain,
$a_{i+k-1}$ is the \emph{last} element of the chain, and the
\emph{length} of the chain is $k$.  We call it a \emph{type-$2$}
chain. We continue recursively to choose all of the points in the
other attractor cycles.  After doing this, the elements of $P'$ are
$a_1,\dots,a_n$ in order.
\medskip

\noindent\textbf{Construction of $\SR$.} We next construct the
associated finite subdivision rule $\cR$.  The $1$-skeleton of the
model subdivision complex $\SR$ is a tree as in Figure
\ref{fig:stickers}.  There is one central vertex.  We identify
$\infty\in\Gamma$ with this central vertex.  There are $n$
``stickers'' (a sticker is an edge of the graph with a vertex of
valence one, like a stick pin with a spherical head) from $\infty$ to
valence $1$ vertices $a_1,\dots,a_n$, in counterclockwise order.  We
identify $a_1,\dots,a_n\in\Gamma$ with these valence $1$ vertices.
The tile type $t$ is a $(2n)$-gon, which we think of as an $n$-gon
with each edge bisected.  The characteristic map $\psi:t\to \SR$ maps
the edge barycenters to the sticker heads and the other vertices to
$\infty$.  The edge barycenters are called finite vertices and the
others are called $\infty$-vertices.  More generally, a vertex of some
subdivision of $t$ which is not a vertex of $t$ is called a finite
vertex.  Every vertex $v$ of $t$ is labeled by $\psi(v)$.  These
vertex labels are placed outside $t$ in the figures.  We use clockwise
order on $\partial t$.

If $k$ is a integer with $k\ge 2$, a
\emph{$k$-doodle} is a graph with three vertices and $k$ edges (none of them
loops) such that one vertex (the central vertex) has valence $k$,
one vertex (the head) has valence $1$, and the third vertex (the foot) has
valence $k-1$. Note that a $2$-doodle is a bisected arc.

We define a subdivision $\cR(t)$ of $t$. We do this in three stages.
We first define a subtiling of $t$ into subtiles such that the finite
vertices in each subtile are in the proper cyclic order.  This means
that there might be fewer than $n$ of them, but they will have image
labels, which are distinct elements of $\{a_1,...,a_n\}$, and, when
taken in clockwise order, their image labels have the same cyclic
order as in $(a_1,...,a_n)$.  We then add arcs and $k$-doodles as
determined by the critical vertices of $\Gamma$ so that we have $d$
subtiles.  Finally, we add stickers as necessary to get the
subdivision $\cR(t)$.  The tiles of the first stage will be defined so
that the resulting subdivision map does not have Levy cycles.  We do
this by ensuring that there is an iterated subdivison $\cR^n(t)$ of
$t$ such that for each finite vertex $v$ of $t$ except possibly one,
either there is a new edge from $v$ to an $\infty$-vertex or there is
an arc (made out of two or four new edges) from the $\infty$-vertex
before $v$ to the $\infty$-vertex after $v$.

As we construct $\cR(t)$, we will give \emph{image labels} to its
vertices. The image label of a vertex is the vertex it will map to
under the analog of $\psi$ from $\cR(t)$ to $\SR$.  So for a vertex
$v$ in $t$, the image label of $v$ is defined to be $\tau(\psi(v))$.
We will keep track of the critical vertices that have already been
accounted for during the construction of $\cR(t)$ in a set
$\widetilde{C}$. For the beginning of the construction, we define
$\widetilde{C} = \emptyset$.
\medskip

\noindent\textbf{Stage 1.} Suppose $a_i$ is the last element of a
chain, and $a_j$ is the first element of the next chain (in cyclic
order). So either $1\le i < n$ and $j=i+1$ or $i=n$ and $j=1$.  The
construction in stage 1 depends on the types of the chains which
contain $a_i$ and $a_j$.  We consider various cases.

If $a_i$ and $a_{j}$ are in distinct chains of type $1$ or if they are
in the same chain of type $1$ (there is only one chain) and
$\tau(a_i)\ne a_j$, then we add a $k$-doodle, with $k$ being the
degree of the critical vertex $c_{a_{j}}$, with head the
$\infty$-vertex after $a_i$ and with foot the $\infty$-vertex before
$a_i$. We give the central vertex of the $k$-doodle image label
$a_{j}$, and add $c_{a_{j}}$ to $\widetilde{C}$.  See Figure
\ref{fig:type1type1}, which, like Figures
\ref{fig:type1type2}--\ref{fig:type1cycle}, is drawn with $k=2$.

If $a_i$ is in a chain of type $1$ and $a_{j}$ is in a chain of type $2$,
then we add $k-1$ edges joining $a_j$ to
the $\infty$-vertex before $a_i$ (where $k = \deg(a_j)$)
and add $a_{j}$ to $\widetilde{C}$.
See Figure \ref{fig:type1type2}.

Suppose $a_i$ is in a chain of type $2$ and $a_j$ is in a chain of type $1$
(this can only occur if $i=n$ and $j=1$). If $a_i$ is in a chain of length $1$,
then we don't do anything at this stage.  
If $a_i$ is in a chain of length greater
than $1$, then we add a $k$-doodle 
(with $k = \deg(c_{a_{j}})$) with head the $\infty$-vertex after $a_i$ and with 
foot the $\infty$-vertex before $a_i$. We give the central vertex
of the $k$-doodle image label $a_{j}$, and add $c_{a_{j}}$ to
$\widetilde{C}$.
Figure \ref{fig:type2type1} shows both possibilities.

If $a_i$ and $a_{j}$ are both in chains of type $2$ and $a_i$ is in a
chain of length $1$, then to $a_{j}$ we add $k-1$ edges joining it to
the $\infty$-vertex before $a_{j}$ (where $k= \deg(a_j)$) and add
$a_{j}$ to $\widetilde{C}$.  If $a_i$ and $a_{j}$ are both in chains
of type $2$ and $a_i$ is in a chain of length greater than $1$, then
to $a_{j}$ we add $k-1$ edges joining it to the $\infty$-vertex before
$a_{i}$ (where $k=\deg(a_j)$) and add $a_{j}$ to $\widetilde{C}$.  The
two possibilities are shown in Figure \ref{fig:type2type2}.

Now suppose that there is a single chain, it has type $1$, and
$\tau(a_n) = a_1$. This is the only remaining case.  The
Riemann-Hurwitz condition implies that if there is just one finite
critical value, then $\Gamma$ has only two critical vertices and their
degrees both equal the degree of $\Gamma$.  Hence the finite critical
vertex is the only vertex of $\Gamma$ which $\tau$ maps to the finite
critical value.  This is impossible in the present case because
$a_1\in A$ and $\tau(a_n)=a_1$.  So either one of the $a_i$'s is a
critical vertex or one of the $a_i$'s with $i>1$ is a critical value.

First suppose that $r\in \{1,\dots,n\}$ and $a_r$ is a critical vertex
with degree $k$. We add $k-1$ arcs in $t$ from $a_r$ to the
$\infty$-vertex before $a_r$, and we add $a_r$ to $\widetilde{C}$. See
the left side of Figure \ref{fig:type1cycle}.

If none of the $a_i$'s is a critical vertex, then some $a_i$ with
$i>1$ is a critical value. In this case, suppose $r\in\{2,\dots,n\}$
such that $a_r$ is a critical value.  Let $k_1 = \deg(c_{a_1})$, and
let $k_r = \deg(c_{a_r})$.  We add a $k_r$-doodle with head the
$\infty$-vertex after $a_n$ and with foot the $\infty$-vertex before
$a_n$. We give its central vertex image label $a_r$, and we add
$c_{a_r}$ to $\widetilde{C}$.  We then add a $k_1$-doodle with head
the $\infty$-vertex after $a_n$ and with foot the $\infty$-vertex
before $a_n$ as indicated in Figure \ref{fig:type1cycle}.  We give its
central vertex image label $a_1$, and we add $c_{a_1}$ to
$\widetilde{C}$.  See the right side of Figure \ref{fig:type1cycle}.
This completes stage $1$ of the construction.

\begin{figure}[!ht]
\labellist
\small\hair 2pt
\pinlabel $a_i$ at 29 26
\pinlabel $a_j$ at 103 23
\pinlabel $\textcolor{darkgray}{a_j}$ at 50 1
\pinlabel $\textcolor{darkgray}{\tau(a_i)}$ at 40 14
\pinlabel $\textcolor{darkgray}{\tau(a_j)}$ at 84 14
\endlabellist
\centering
\includegraphics{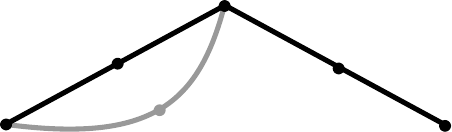}
\caption{One or two chains of type $1$}
\label{fig:type1type1}
\end{figure}
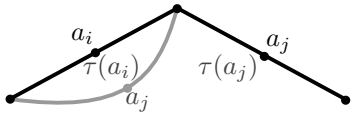

\begin{figure}[!ht]
\labellist
\small\hair 2pt
\pinlabel $a_i$ at 29 26
\pinlabel $a_j$ at 103 23
\pinlabel $\textcolor{darkgray}{\tau(a_i)}$ at 40 13
\pinlabel $\textcolor{darkgray}{\tau(a_j)}$ at 98 7
\endlabellist
\centering
\includegraphics{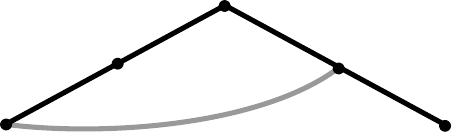}
\caption{A chain of type $1$ followed by a chain of type $2$}
\label{fig:type1type2}
\end{figure}

\begin{figure}[!ht]
\labellist
\small\hair 2pt
\pinlabel $a_n$ at 29 26
\pinlabel $a_1$ at 103 23
\pinlabel $\textcolor{darkgray}{a_n}$ at 39 14
\pinlabel $a_n$ at 197 26
\pinlabel $a_1$ at 271 23
\pinlabel $\textcolor{darkgray}{a_1}$ at 215 0
\pinlabel $\textcolor{darkgray}{\tau(a_n)}$ at 209 15
\pinlabel $\textcolor{darkgray}{\tau(a_1)}$ at 252 15
\endlabellist
\centering
\includegraphics{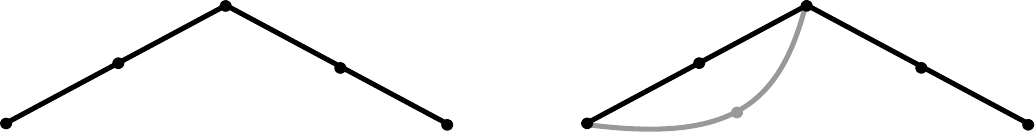}
\caption{A chain of type $2$ followed by a chain of type $1$}
\label{fig:type2type1}
\end{figure}

\begin{figure}[!ht]
\labellist
\small\hair 2pt
\pinlabel $a_i$ at 29 26
\pinlabel $a_j$ at 103 23
\pinlabel $\textcolor{darkgray}{a_i}$ at 39 13
\pinlabel $a_i$ at 197 26
\pinlabel $a_j$ at 271 23
\pinlabel $\textcolor{darkgray}{\tau(a_j)}$ at 95 9
\pinlabel $\textcolor{darkgray}{\tau(a_i)}$ at 209 15
\pinlabel $\textcolor{darkgray}{\tau(a_j)}$ at 264 8
\endlabellist
\centering
\includegraphics{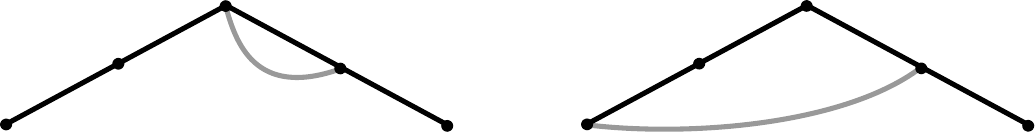}
\caption{Two chains of type $2$}
\label{fig:type2type2}
\end{figure}

\begin{figure}[!ht]
\labellist
\small\hair 2pt
\pinlabel $a_r$ at 103 27
\pinlabel $a_n$ at 197 30
\pinlabel $a_1$ at 271 27
\pinlabel $\textcolor{darkgray}{a_r}$ at 220 10
\pinlabel $\textcolor{darkgray}{a_1}$ at 227 1
\pinlabel $\textcolor{darkgray}{\tau(a_r)}$ at 94 16
\pinlabel $\textcolor{darkgray}{\tau(a_1)}$ at 259 17
\pinlabel $\textcolor{darkgray}{a_1}$ at 209 21
\endlabellist
\centering
\includegraphics{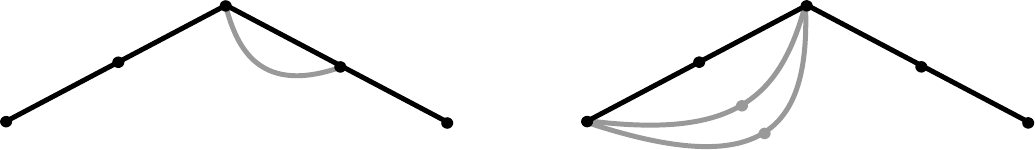}
\caption{A single chain and it has type $1$}
\label{fig:type1cycle}
\end{figure}
\medskip

\noindent\textbf{Verification that image labels are consistent after
Stage 1.} We now look at what we have after stage $1$.  Every subtile
except for the central one is either a $2$-gon or a $4$-gon, and so
there are only one or two finite vertices.  For a $2$-gon there is
only one finite vertex and so its image label is in proper cyclic
order.

For a $4$-gon, there are two finite vertices, so their image labels
are in proper cyclic order if they are distinct.  The only potential
problem is if, in the notation of Figure \ref{fig:type1type1},
$\tau(a_i) = a_j$.  Suppose that this happens.  Then $a_i$ and $a_j$
are in different chains.  Because $\tau(a_i) = a_j$ and $a_i$ and
$a_j$ are in different chains, $a_j$ is not periodic under $\tau$.
But if there exists such a vertex when a chain is defined, then the
first vertex of that chain must not be the image of a postcritical
vertex.  So it is not possible that $\tau(a_i) = a_j$.  Hence the two
finite vertices of every 4-gon have different image labels.

  Now we verify that the same is true for the central tile $s$.
Suppose that $a_i,...,a_k$ are the vertices of a chain in order.  Then
$a_i,...,a_k$ are labels of consecutive finite vertices $v_i,...,v_k$
of $t$.  Moreover, $v_i,...,v_{k-1}$ are consecutive finite vertices
of $s$.  Their image labels are $a_{i+1},...,a_k$.  In the cases
corresponding to Figure \ref{fig:type1type1}, Figure
\ref{fig:type1cycle} and the right half of Figure
\ref{fig:type2type1}, the finite vertex of $s$ preceeding $v_i$ has
image label $a_i$.  These are the only cases in which $a_i$ is an image
label of a vertex of $s$.  In the situation of
Figure~\ref{fig:type1cycle}, there is only one chain and $\tau(a_n) =
a_1$, so it is clear in this case that the image labels of $s$ are in
proper cyclic order.  In all other cases except those corresponding to
the left halves of Figures~\ref{fig:type2type1} and
\ref{fig:type2type2}, $\tau(a_k)$ is not the image label of a vertex of
$s$.  In the left halves of Figures~\ref{fig:type2type1} and
\ref{fig:type2type2}, we have that $k = 1$ and $\tau(a_k) = a_k$.
Hence the vertices among $a_i,...,a_k$ which are image labels of
vertices of $s$ occur consecutively and in proper order.  Finally, it
is clear that the chains occur in proper order.  So in the central
tile the image labels of the finite vertices are in proper cyclic
order.
\medskip

\noindent\textbf{Stage 2.} For the second stage, we add subtiles
corresponding to the critical vertices in $C'$ that aren't in
$\widetilde{C}$. We do this recursively. Each time we add subtiles
because of an element $c$ of $C' \setminus \widetilde{C}$, we add this
element to $\widetilde{C}$.  Since $C'$ is finite, this process will
terminate.  Suppose $c\in C' \setminus \widetilde{C}$. Let $k =
\deg(c)$.  If $c\in P'$, then we add $k-1$ edges from $c$ to the
$\infty$-vertex of $t$ before $c$, and we add $c$ to $\widetilde{C}$.
Image labels of finite vertices of all tiles remain in proper cyclic
order.  Now suppose $c\not\in P'$, but that there is another element
$c'\in C' \setminus \widetilde{C}$ with $c'\not\in P'$ and
$\tau(c')\ne \tau(c)$.  Let $k = \deg(c)$ and let $k' = \deg(c')$.
Let $i,j\in \{1,\dots,n\}$ such that $\tau(c) = a_i$ and $\tau(c') =
a_j$.  Choose a subtile $s$ of $t$. If $s$ doesn't contain a vertex
with image label $a_i$, then there is a unique $\infty$-vertex in $s$
such that we can add a $k$-doodle with head and tail this
$\infty$-vertex and with image label $a_i$ and still have the image
labels be in cyclic order.  We do this, and we add $c$ to
$\widetilde{C}$.  Image labels of finite vertices of all tiles remain
in proper cyclic order.  Suppose $s$ does contain a vertex $v$ with
image label $a_i$. Then we add a $k'$-doodle to $s$ with head the
$\infty$-vertex after $v$, with central vertex with image label $a_j$,
and with tail the $\infty$-vertex before $v$. We then add a $k$-doodle
to $s$ with head the $\infty$-vertex after $v$, with central vertex
with image label $a_i$, and with tail the $\infty$-vertex before
$v$. There are a new subtile in $s$ with the same image labels (and in
the same cyclic order) as for $s$, some $2$-gons if $k>2$ or $k' > 2$,
and two $4$-gons with a vertex labeled $a_i$ and a vertex labeled
$a_j$. So we still have the image labels of the finite vertices of all
of the subtiles in cyclic order. Finally, we add $c$ and $c'$ to
$\widetilde{C}$.

To complete stage $2$, we need to consider the case that $C'
\setminus \widetilde{C} \ne \emptyset$ and that all elements of
$C' \setminus \widetilde{C}$ have the same image $a_i$ under $\tau$. 
Choose an element $c\in C' \setminus \widetilde{C}$, 
and let $k = \deg(c)$. Let
$r = \#( C' \setminus \widetilde{C} )$ and let $m$ be the sum of the
degrees of the elements of $C' \setminus \widetilde{C}$. 
At every step of the construction thus far, the number of subtiles of $t$ 
increases by $\deg(v)-1$, where $v$ is the vertex added to $\widetilde{C}$.  
So the number of subtiles of $t$ created thus far is
\begin{eqnarray*}
1+\Sigma_{v\in\widetilde{C}}(deg(v)-1)
& = & 1+\Sigma_{v\in C'}(deg(v)-1)-m+r\\
& = & 1+d-1-m+r\\
& = & (d-m)+r.\\
\end{eqnarray*}
Because $\tau$ maps every element of $C'\setminus \widetilde{C}$ to
$a_i$, the number of subtiles that can have a vertex with image label
$a_i$ is at most $d-m$, so there is a subtile that does not have a
vertex with image label $a_i$. There is a unique $\infty$-vertex in
this subtile such that we can add a $k$-doodle with head and tail this
$\infty$-vertex and with image label $a_i$ and still have the image
labels be in cyclic order. We do this, and we add $c$ to
$\widetilde{C}$. This completes the recursive step, so we can continue
the recursion until $C' = \widetilde{C}$. This completes the second
stage. At this point there are $d$ subtiles of $t$, and in each
subtile the image labels of the finite vertices are in proper cyclic
order.
\medskip

\noindent\textbf{Stage 3.}  Suppose $s$ is a subtile of the
construction after stage two, and $i\in \{1,\dots,n\}$. If $s$ doesn't
have a vertex with image label $a_i$, then there is a unique
$\infty$-vertex of $s$ to which we can add a sticker whose other
vertex has image label $a_i$ and still have the image labels of the
finite vertices in cyclic order.  We do this for every such $s$ and
$i$.  This completes stage $3$.
\medskip

\noindent\textbf{Completion of the construction of $\SR$.}  At this point
every subtile has $n$ finite vertices and their image labels are in
proper cyclic order.  We define this to be the subdivision $\cR(t)$,
and we define its image under the characteristic map $t\to\SR$ to be
$\cR(\SR)$. It is straightforward to define a subdivision map that
takes $\cR(\SR)$ to $\SR$, which takes each open cell homeomorphically
to an open cell, takes each $\infty$-vertex to $\infty$, and takes a
vertex with image label $a_i$ to the vertex $a_i$.  This completes the
definition of the finite subdivision rule $\SR$.  It is clear from
this construction that the ramification portrait of the subdivision
map $\sigma_\cR$ is isomorphic to $\Gamma$.  
\medskip

\noindent\textbf{Verification that $\sigma_\cR$ has no Levy cycle.}
We prove by contradiction that $\subm$ cannot have a Levy cycle.
Suppose $\{\delta_0,\dots,\delta_{k-1},\delta_k=\delta_0\}$ is a Levy
cycle for $\subm$.  Choose any $i\in \{1,\dots,k\}$, and consider the
component $D_i$ of $S^2 \setminus \delta_i$ that does not contain
$\infty$. Since $\delta_i$ is essential and is not peripheral, $D_i$
must contain at least two points of $P$.  We will obtain a
contradiction by showing that $D_i$ can contain at most one point of
$P$.  For this, consider a finite postcritical point $a_j$ of $\subm$
in $\SR$.

Suppose that there exists a positive integer $m$ such that
$\tau^{\circ m}(a_j)$ is a critical vertex.  Then $a_j$ and $\infty$
are joined by a new edge in $\cR^{m+1}(\SR)$.  By Lemma
\ref{lem:Levyedge} we can isotop $\delta_i$ in $S^2\setminus P_f$ to
be disjoint from this new edge, so $a_j$ and $\infty$ are in the same
component of $S^2 \setminus \delta_i$.  Hence $a_j\notin D_i$.  So
$a_j\notin D_i$ if either $a_j$ is in a type-$2$ chain or we are in
the situation of the left half of Figure \ref{fig:type1cycle}.

Now suppose that $a_j$ is in a type-$1$ chain that is not followed by
a type-$2$ chain. Let $a_r$ be the last element of the type-$1$ chain
that contains $a_j$. We are in the situation of either
Figure~\ref{fig:type1type1} or the right half of
Figure~\ref{fig:type1cycle}.  So there is a pair of new edges of
$\cR(\SR)$ that bounds an open disk $D$ that contains $a_r$ and no
other postcritical points.  Hence $a_r$ cannot be in the open disk
$D_i$ since if so we can isotop $D_i$ rel $P_f$ into $D$.  Similarly,
in $\cR^{r-j+1}(\SR)$ there is a pair of new edges that bounds an open
disk that contains $a_j$ and no other postcritical points.

Finally, suppose $a_j$ is in a type-$1$ chain which is followed by a
type-$2$ chain.  This is the only remaining possibility.  Let $a_r$ be
the last element of the type-$1$ chain containing $a_j$.  Suppose that
$j\ne r$.  Let $u$ be the vertex of $t$ with label $a_j$.  Let $t'$ be
the tile of $\cR^{r-j}(t)$ which contains $u$.  Then the label of $u$
relative to $t'$ is $a_r$, that is, the structure map of $t'$ from
$t'$ to $S_\cR$ maps $u$ to $a_r$.  Let $v$ be the finite vertex of
$t'$ following $u$.  The label of $v$ relative to $t'$ is $a_{r+1}$.
Because (i) $j\ne r$, (ii) $a_r$ is in a type-1 chain and (iii)
$a_{r+1}$ is in a type-2 chain, the definition of chains implies that
$v$ is not the finite vertex of $t$ following $u$.  So the edge $e_1$
of $t'$ joining $v$ and the $\infty $-vertex of $t'$ following $u$
must be a new edge.  But, as in Figure \ref{fig:type1type2}, there is
a new edge $e_2$ in $\cR(t')$ joining $v$ and the $\infty$-vertex of
$t'$ preceding $u$.  As before, the two new edges in
$\cR^{r-j+1}(\SR)$ corresponding to $e_1$ and $e_2$ bound an open disk
which contains $a_j$ and no other postcritical point.  We have reduced
to the case in which $j=r$.  Chains are defined so that at most one
postcritical point has this property.  So the open disk $D_i$ can
contain at most one postcritical point, which contradicts the
assumption that $\delta_i$ is an element of a Levy cycle.

Since $\sigma_\cR$ is a Thurston map whose ramification portrait is
isomorphic to $\Gamma$ and $\sigma_\cR$ has no Levy cycle, the proof
of Theorem 1 is complete.
\end{proof}

\section{Completely unobstructed portraits}\label{sec:unobstructed}

\begin{proof}[Proof of Theorem \ref{thm:unobstructed}] Suppose
$\Gamma$ is an abstract polynomial portrait of degree $d$ that has at
least four vertices and satisfies condition (i) or (ii) of the
statement of the theorem.  Suppose $f$ is a Thurston map with portrait
isomorphic to $\Gamma$, and with $\infty$ a fixed critical point such
that $\deg_f(\infty) = d$.  We prove by contradiction that $f$ is
unobstructed.

Suppose $f$ is obstructed. Then $f$ has a Levy cycle, and (in the
terminology of \cite[Section 10.3]{Hub}) $f$ has a degenerate Levy
cycle $\{\delta_0,\dots,\delta_{n-1},\delta_n = \delta_0\}$.  This
means the following.  For each $i\in \{0,\dots,n\}$, let $D_i$ be the
disk bounded by $\delta_i$ in the 2-sphere such that $\infty\not\in
D_i$.  For each $i\in \{1,\dots,n\}$, one component of $f^{-1}(D_i)$
is a disk $D_{i-1}'$ such that
\begin{enumerate}
\item[(a)] $D_i\cap D_j=\emptyset $ if $i\ne j\in \{1,\dotsc,n\}$
\item[(b)] the boundary of $D_{i-1}'$ is isotopic to $\delta_{i-1}$ rel $P_f$, 
\item[(c)] $D_{i-1}' \cap P_f = D_{i-1}\cap P_f$, and
\item[(d)] $f|\co D_{i-1}' \to D_i$ is a homeomorphism.
\end{enumerate}

A key point for the Levy-Berstein theorem is that a postcritical point
in one of the $D_i$'s cannot be a critical point, because that would
violate d).  But it also cannot have an iterate that is a critical
point, because that would imply that some $D_j$ contains a critical
point.  Since each $D_i$ must contain at least two postcritical
points, there must be at least two postcritical points in
non-attractor cycles. This gives the contradiction for case (i).

Now suppose (ii) holds.  Then for some prime number $p$ and positive
integer $k$, we can enumerate the finite postcritical points of $f$ as
$\{v_i: 0\leq i < p^k\}$ such that $f(v_i) = v_{i+1}$ (mod $p^k$) for
every $i\in \{0,\dots,p^k -1\}$, and if $v_j$ is a critical value then
$j$ is a multiple of $p^{k-1}$.  Since the sets $P_f\cap D_i$
partition the set of finite postcritical points and they all have the
same cardinality, there is a positive integer $m$ such that $\#(P_f
\cap D_i) = p^m$ for all $i$.  Then $np^m = p^k$ and $n = p^r$, where
$r = k-m$. For some $i\in \{1,\dots,n\}$, $v_0\in D_i$.  Then
$\{v_{jp^r}\co 0\le j < p^m\} \subset D_i$ and so $\{v_{jp^{k-1}}\co
0\le j < p\} \subset D_i$.  Thus $D_i$ contains every finite critical
value of $f$.

Let $D$ be the disk bounded by $\delta_i$ that contains $\infty$, and
let $\widetilde{D} = f^{-1}(D)$.  Then $D$ doesn't contain any finite
critical values of $f$.  It follows that the restriction of $f$ to
$\widetilde{D}\setminus \{\infty \}$ is a covering map onto
$D\setminus \{\infty \}$.  But every connected covering space of a
once-punctured disk is a once-punctured disk.  Since $f$ is $d$-to-1
near $\infty $, the space $\widetilde{D}\setminus \{\infty \}$ is a
once-punctured disk which maps by $f$ to $D\setminus \{\infty \}$ with
degree $d$.  Hence $\partial \widetilde{D}=f^{-1}(\partial
D)=f^{-1}(\delta_i)$.  This contradicts the assumption that
$f^{-1}(\delta_i)$ has a connected component which maps to $\delta_i$
with degree 1.  Thus $f$ is unobstructed.
\end{proof}

\section{Rose maps}\label{sec:rose}

To prove Theorem \ref{thm:obstructed}, we need a topological
description for topological polynomials which may not be subdivision
maps.  We begin this section by discussing our approach to this.  We
define a \emph{rose} to be the boundary of the union of finitely many
closed topological disks in the 2-sphere which are disjoint except for
having exactly one point in common.  We view a rose as a graph with
exactly one vertex.  Its edges are called \emph{petals}.  

Let $S_1^2$ and $S_2^2$ be two copies of $S^2$, and suppose that we
have a finite branched covering map $g\co S_1^2\to S_2^2$ whose
critical values lie in a finite set $P\subseteq S_2^2$.  The
restriction of $g$ to $S_1^2\setminus g^{-1}(P)$ is a covering map
from $S_1^2\setminus g^{-1}(P)$ onto $S_2^2\setminus P$.  In the
context of covering maps, a straightforward thing to do in this
situation is to use the fact that $S_2^2\setminus P$ is homotopic to a
rose with $\#(P)-1$ petals---the fundamental group of $S_2^2\setminus
P$ is a free group on $\#(P)-1$ generators.  Let $R_2$ be a rose in
$S_2\setminus P$ which is a spine, and let $R_1=g^{-1}(R_2)$.  Because
every connected component of $S_2^2\setminus R_2$ contains at most one
branch value of $g$, every connected component of $S_1^2\setminus R_1$
is a disk, equivalently, $R_1$ is connected.  The restriction of $g$
to $R_1$ is a covering map onto $R_2$, and this restriction determines
$g$ up to homotopy.

With this in mind, we construct finite branched covering maps $g\co
S_1^2\to S_2^2$ as follows.  Let $R_2\subseteq S_2^2$ be a rose with
$n\ge 1$ petals.  We orient $S_2^2$ and label $n$ connected components
of $S_2^2\setminus R_2$ each bounded by one petal with $1,\dotsc,n$ in
counterclockwise order.  We label the remaining connected component of
$S_2^2\setminus R_2$ with $\infty $.  Suppose that we have a finite
connected graph $R_1\subseteq S_1^2$ whose vertices have small
neighborhoods which look like small neighborhoods of the vertex of
$R_2$.  Here is what this means.  The connected components of
$S_1^2\setminus R_1$ are labeled with $1,\dotsc,n$ (duplications
allowed) and $\infty $.  We choose a barycenter for every edge of
$R_1$ and call the resulting edges half edges.  These barycenters are
not vertices of $R_1$.  Let $v$ be a vertex of $R_1$.  Then $2n$ half
edges contain $v$.  After orienting $S_1^2$, they can be written as
$\epsilon_1,\dotsc,\epsilon_{2n}$ in counterclockwise order around $v$
so that $\epsilon_{2i-1}$ and $\epsilon_{2i}$ are in the boundary of a
connected component of $S_1\setminus R_1$ with label $i$ for every
$i\in\{1,\ldots,n\}$.  Furthermore $\epsilon_{2i}$ and
$\epsilon_{2i+1}$ are in the boundary of a connected component of
$S_1^2\setminus R_1$ with label $\infty $ for every
$i\in\{1,\ldots,n\}$, where $\epsilon_{2n+1}=\epsilon_1$.  It is a
straightforward matter, by mapping vertices, then half edges and then
disks, to construct a finite branched covering map $g\co S_1^2\to
S_2^2$ such that $g|_{R_1}$ is a covering map onto $R_2$ which maps
vertices to vertices and edges to edges.  This can be done so that $g$
has at most one critical point in every connected component of
$S_1^2\setminus R_1$.

We define a rose map to be a map of pairs $g\co (S_1^2,R_1)\to
(S_2^2,R_2)$, where $S_1^2$ and $S_2^2$ are two oriented copies of
$S^2$, $g\co S_1^2\to S_2^2$ is an orientation-preserving finite
branched covering map, $R_2\subseteq S_2^2$ is a rose, $R_1=
g^{-1}(R_2)$ is a graph with pullback graph structure and every
connected component of $S_2^2\setminus R_2$ contains at most one
critical value of $g$.  The next lemma guarantees the existence of the
rose maps that we will use for the proof of Theorem
\ref{thm:obstructed}.  We will precompose a rose map $g\co S_1^2\to
S_2^2$ with a homeomorphism $h\co S_2^2\to S_1^2$ to obtain a desired
topological polynomial $f=g\circ h$.  In the lemma, the connected
components of $S_2^2\setminus R_2$ are labeled, which induces a
labeling of the connected components of $S_1^2\setminus R_1$, which
induces a labeling of the vertices of the graph dual to $R_1$.

\begin{lemma}\label{lem:cmbllemma} Suppose that $\Gamma$ is an
abstract polynomial portrait whose finite postcritical vertices are
$v_1,\dots,v_n$.  Let $u$ be a finite critical value of $\Gamma$ with
the maximum number of incoming edges from critical vertices.  Let $v$
be any finite critical value of $\Gamma$.  Then there exists a rose
map $g\co (S_1^2,R_1)\to (S_2^2,R_2)$ realizing the branch data of
$\Gamma$ such that $n$ connected components of $S_2^2\setminus R_2$
each bounded by a petal of $R_2$ are labeled $v_1,\dots,v_n$ in
counterclockwise order, the remaining connected component is labeled
$\infty$ and $R_1$ has a dual graph $R_1^*$ for which the following
statements hold.
\begin{enumerate}
  \item The boundary of one connected component of $S_1^2\setminus
R_1^*$ contains exactly two critical points: one vertex with label $u$
and one vertex with label $\infty $.
  \item If $u$ has an incoming edge from a noncritical vertex, then
the boundary of one connected component of $S_1^2\setminus R_1^*$
contains exactly two critical points: one
vertex with label $v$ and one vertex with label $\infty $.
\end{enumerate}
\end{lemma}
  \begin{proof} We will construct $R_1$ rather explicitly.

To prepare for the construction of $R_1$, let $U$ be the set of
critical vertices which $\tau$ maps to $u$, and let $W$ be the set of
remaining finite critical vertices.  So $C_\Gamma=U\amalg W\amalg
\{\infty\}$.  The Riemann-Hurwitz condition gives that
  \begin{equation*}
\sum_{w\in U}^{}(\deg_\tau(w)-1)+\sum_{w\in W}^{}(\deg_\tau(w)-1)=d-1,
  \end{equation*}
where $d$ is the degree of $\Gamma$.  So
  \begin{equation*}\linnum\label{line:rh}
\sum_{w\in W}^{}(\deg_\tau(w)-1)=d-1+\#(U)-\sum_{w\in U}^{}\deg_\tau(w)\ge
\#(U)-1.
  \end{equation*}
In particular, if $W$ is empty, then $\#(U)=1$.  In this case there is
exactly one choice for $R_1$ up to isomorphism.  Only two connected
components of $S_1^2\setminus R_1$ are not monogons, and one of
these has label $\infty$.  Statement 1 is true in this case, and
statement 2 is true since $u=v$.  So we henceforth assume that $W$ is
nonempty.

Let $m$ be the maximum number of incoming edges from critical vertices
at the critical values other than $u$.  Then it is possible to
partition $W$ into disjoint nonempty subsets $W_1,\dotsc,W_m$ so that
if $w,w'\in W_i$ for some $i$ with $w\ne w'$, then $\tau(w)\ne
\tau(w')$.  We do this so that if $u\ne v$, then there exists $w\in
W_m$ such that $\tau(w)=v$.

Now we begin to construct $R_1$.  We enumerate the elements of $W_1$,
and for every $w\in W_1$ we construct a closed (2-dimensional) polygon
$P_w\subseteq S_1^2$ with $\deg_\tau(w)$ sides whose interior has
label $\tau(w)$.  These polygons are disjoint from each other, except
that each has exactly one vertex in common with the next.  We obtain a
chain (not to be confused with the chains in Section~\ref{sec:poly})
$C_1$ of polygons.  We also construct such chains $C_2,\dotsc,C_m$ for
$W_2,\dotsc,W_m$ so that the chains $C_1,\dotsc,C_m$ are disjoint from
each other and if $u\ne v$, then the polygon with label $v$ in $C_m$
is last.

The choices of $u$ and $m$ imply that $U$ contains at least $m$
elements.  We choose $m-1$ distinct elements $u_1,\dotsc,u_{m-1}\in
U$.  For every $i\in\{1,\ldots,m-1\}$ we construct a polygon $P_{u_i}$
as before which is disjoint from the polygons already constructed and
from the other $P_{u_j}$'s,
except that one vertex of $P_{u_i}$ is a vertex of $C_i$ that is only
in the last polygon of $C_i$, and a different vertex of $P_{u_i}$ is a
vertex of $C_{i+1}$ that is only in the first tile of $C_{i+1}$.  
The polygons
$P_{u_1},\dotsc,P_{u_{m-1}}$ join the chains $C_1,\dotsc,C_m$ to form
a single chain $C$.

We intend to also construct similar polygons $P_v$ for the remaining
elements $v$ of $U$.  We intend to construct each of them in one of
two ways.  One way to construct $P_v$ is to choose a vertex of $C$
contained in only one polygon and to construct $P_v$ so that it meets
the polygons already constructed exactly in this vertex.  Here is
another way to construct $P_v$.  Choose $i\in\{1,\ldots,m\}$ and two
consecutive polygons $P$ and $P'$ in $C_i$.  Let $x$ be the vertex
common to $P$ and $P'$.  We modify $P$ and $P'$ slightly near $x$,
pulling them apart, so that they become disjoint.  We then construct
$P_v$ so that it contains both of the new vertices in $P$ and $P'$
while being otherwise disjoint from the polygons already constructed.

The only obstacle to performing the constructions described in the
previous paragraph is that $C$ might not contain enough vertices to
accommodate all the elements of $U$.  But it is not difficult to see
that the number of elements of $U$ that can be accommodated, including
$u_1,\dotsc,u_{m-1}$, is
  \begin{equation*}
1+\sum_{w\in W}(\deg_\tau(w)-1).
  \end{equation*}
Thus line~\ref{line:rh} shows that $C$ does indeed have enough
vertices to accommodate all the elements of $U$.  So we construct a
polygon $P_v$ as described in the previous paragraph for every $v\in
U\setminus \{u_1,\dotsc,u_{m-1}\}$.

Because $U$ contains at least $m$ elements, this can be done so that
  \begin{equation*}\linnum\label{line:condition1}
\parbox{.85\linewidth}{the first of these polygons meets exactly one 
polygon in $C$, the
first polygon in $C_1$.}
  \end{equation*}
Furthermore, if $u$ has an incoming edge from a noncritical vertex,
then the inequality in line~\ref{line:rh} is strict, so we may also
construct these polygons so that
  \begin{equation*}\linnum\label{line:condition2}
\parbox{.85\linewidth}{if $u$ has an incoming edge from a noncritical
vertex, then the last polygon in $C_m$ contains a vertex not in any of
these polygons.}
  \end{equation*}

Every vertex in the complex constructed thus far is contained in
either one or two polygons.  If there are two, then their labels are
different.  Hence it is possible to add monogons with labeled interiors
at each of
these vertices so that the cyclic order of labels about each vertex
agrees with the cyclic order of the labels of $R_2$.  We have $R_1$.
The discussion preceeding the lemma describes how to construct a rose map
$g\co (S_1^2,R_1) \to (S_2^2,R_2)$ from this information.
Line~\ref{line:condition1} implies statement 1 of the lemma, and
line~\ref{line:condition2} implies statement 2.

\end{proof}

\section{Obstructed portraits}\label{sec:obstructed}

\begin{proof}[Proof of Theorem \ref{thm:obstructed}] We first prove
the theorem in cases (i) and (ii).  Assume that $\Gamma$ satisfies
either (i) or (ii).  We will use the following example to illustrate
various constructions in the proof.  Consider the abstract polynomial
portrait that is shown below.

$$\xymatrix{\ast\ar[dr]^{3}&&&\ast\ar[d]^ {2}& \ast\ar[dr] ^{2}&
\ast\ar[d]^2 \\
\ast\ar[r]^2&v_{0}\ar[r] &v_{1}\ar[r] &v_{2} \ar[r] &v_{3} \ar[r] &
v_{4}\ar[d] & \infty\ar@(ur,dr)^8\\
& & v_8\ar[ul] & v_7 \ar[l] & v_6 \ar[l] & v_5 \ar[l] & \ast \ar[l]_{2}}
$$
It satisfies the conditions of case (i).

We prepare to apply Lemma \ref{lem:cmbllemma}.  Suppose that $\Gamma$
has $n$ finite postcritical vertices $v_0,\dotsc,v_n$ with $v_n=v_0$
such that $\tau(v_i)=v_{i+1}$ for every $i\in\{0,\ldots,n-1\}$.  Of
course, if $\Gamma$ satisfies condition (i), then $n=p^k$.  Choose a
finite critical value $u$ of $\Gamma$ with the maximum number of
incoming edges from critical vertices.  We redefine $v_0,\dotsc,v_n$
if necessary so that $v_0=u$ without changing the assumptions.

In particular, if $\Gamma$ satisfies (i), then there exists a critical
value $v_i\in \{v_0,\dotsc,v_{p^k-1}\}$ such that $i$ is not a
multiple of $p^{k-1}$.  We set $v=v_i$.  In our example, we take
$u=v_0$ and $v=v_2$.

To define $v$ in the case of condition (ii), suppose that $\Gamma$
satisfies (ii).  Because $u$ is in a non-attractor cycle, $\tau$ maps
some vertex which is not critical to $u$.  It follows that the
inequality in line \ref{line:rh} is strict, and so $\Gamma$ has more
than one finite critical value.  Let $v$ be any finite critical value
other than $u$.

We next define a positive integer $m$.  If $\Gamma$ satisfies (i),
then we set $m=p^{k-1}$.  Suppose that $\Gamma$ satisfies (ii).  Let
$i$ be the index such that $v=v_i$.  Because $n$ is not a prime power,
it is the product of two relatively prime proper divisors.  Because
they are relatively prime, if both of these two proper divisors divide
$i$, then $n$ divides $i$, which is not true.  Hence some positive
proper divisor of $n$ does not divide $i$.  Let $m$ be such a positive
proper divisor of $n$.

So in either case (i) or (ii), $m$ is a positive proper divisor of $n$
such that $v=v_i$ and $m$ does not divide $i$.

Lemma \ref{lem:cmbllemma} implies that there exists a rose map $g\co
(S_1^2,R_1)\to (S_2^2,R_2)$ realizing the branch data of $\Gamma$ such
that $n$ connected components of $S_2^2\setminus R_2$ each bounded by
a petal of $R_2$ are labeled $v_0,\dotsc,v_{n-1}$ in counterclockwise
order, the remaining connected component is labeled $\infty $ and
$R_1$ has a dual graph $R_1^*$ for which
\begin{enumerate}
  \item the boundary of one connected component $C_u$ of
$S_1^2\setminus R_1^*$ contains exactly two critical points: one
vertex with label $u$ and one vertex with label $\infty $;
  \item the boundary of one connected component $C_v$ of $S_1^2\setminus
R_1^*$ contains exactly two critical points: one vertex with label $v$
and one vertex with label $\infty $.
\end{enumerate}
By modifying $R_1^*$ if necessary, we may assume that the restriction
of $g$ to both $C_u$ and $C_v$ is injective.  We identify every
postcritical vertex $w$ of $\Gamma$ with a point in the connected
component of $S_2^2\setminus R_2$ with label $w$.  These serve as the
vertices of a graph $R_2^*$ dual to $R_2$.  Their $g$-pullbacks serve
as the vertices of $R_1^*$.  Figure \ref{fig:rosep3} depicts important
features of the rose map $g$ for our example.

\begin{figure}[!ht]
\labellist
\small\hair 2pt
\pinlabel $v_0$ at 301 59
\pinlabel $v_1$ at 289 53
\pinlabel $v_2$ at 282 36
\pinlabel $v_3$ at 287 20
\pinlabel $v_4$ at 301 11
\pinlabel $v_5$ at 319 14
\pinlabel $v_6$ at 328 28
\pinlabel $v_7$ at 327 44
\pinlabel $v_8$ at 319 57
\pinlabel $v_1$ at 75 158
\pinlabel $v_2$ at 62 167
\pinlabel $v_3$ at 44 163
\pinlabel $v_4$ at 35 151
\pinlabel $v_5$ at 34 136
\pinlabel $v_6$ at 44 122
\pinlabel $v_7$ at 60 120
\pinlabel $v_8$ at 75 128
\pinlabel $v_0$ at 92 140
\pinlabel $v_0$ at 199 168
\pinlabel $v_1$ at 184 158
\pinlabel $v_2$ at 177 143
\pinlabel $v_3$ at 185 127
\pinlabel $v_4$ at 200 119
\pinlabel $v_5$ at 216 122
\pinlabel $v_6$ at 224 135
\pinlabel $v_7$ at 226 152
\pinlabel $v_8$ at 215 164
\pinlabel $\infty$ at 128 72
\pinlabel $S^{2}_{1}$ at 20 84
\pinlabel $S^{2}_{2}$ at 240 0
\pinlabel \rotatebox{-45}{$\underrightarrow{g}$} at 250 87
\endlabellist
\centering
\includegraphics{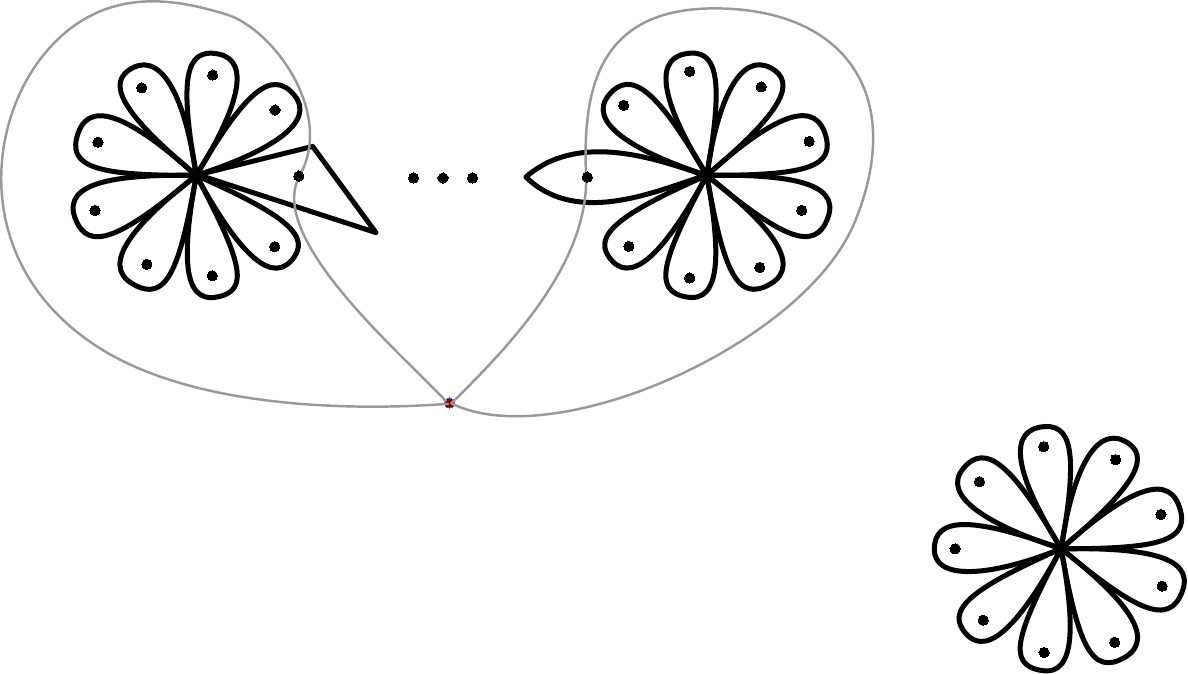}
\caption{Rose map $g:S_{1}^2\to S_{2}^2$}
\label{fig:rosep3}
\end{figure}

Now we choose $m$ disjoint closed topological disks $D_i\subseteq
S_2^2\setminus \{\infty \}$ such that i) $v_j$ is in the interior of
$D_i$ if $j\equiv i\text{ (mod }m)$ for $i\in \{0,\dotsc,m-1\}$ and
$j\in \{0,\dotsc,n-1\}$, ii) $D_0$ is in the open disk $g(C_v)$, and
iii) for $i\in \{1,\dotsc,m-1\}$ $D_i$ is in the open disk
$g(C_u)$.  The restriction of $g$ to $C_v$ is a
homeomorphism, so there exists a unique lift $\widetilde{D}_0$ of
$D_0$ to $C_v$.  We denote the lift of $v_j$ to $C_v$ by
$\widetilde{v}_j$ for every index $j\equiv 0\text{ (mod } m)$.  There
likewise exist unique lifts $\widetilde{D}_i$ of $D_i$ to $C_u$ for
every $i\in \{1,\dotsc,m-1\}$.  We denote the lift of $v_j$ to $C_u$
by $\widetilde{v}_j$ for every index $j\not\equiv 0\text{ (mod }m)$.
Figure \ref{fig:rosep4}  shows all of these points and disks for our example.

\begin{figure}[!ht]
\labellist
\small\hair 2pt
\pinlabel $v_0$ at 300 66
\pinlabel $v_1$ at 286 58
\pinlabel $v_2$ at 282 43
\pinlabel $v_3$ at 285 25
\pinlabel $v_4$ at 300 16
\pinlabel $v_5$ at 319 19
\pinlabel $v_6$ at 330 35
\pinlabel $v_7$ at 330 54
\pinlabel $v_8$ at 319 62
\pinlabel $\widetilde{v}_1$ at 75 165
\pinlabel $\widetilde{v}_2$ at 61 170
\pinlabel $v_3$ at 44 168
\pinlabel $\widetilde{v}_4$ at 35 156
\pinlabel $\widetilde{v}_5$ at 35 141
\pinlabel $v_6$ at 44 127
\pinlabel $\widetilde{v}_7$ at 61 126
\pinlabel $\widetilde{v}_8$ at 72 135
\pinlabel $v_0$ at 95 148
\pinlabel $\widetilde{v}_0$ at 199 171
\pinlabel $v_1$ at 180 164
\pinlabel $v_2$ at 176 147
\pinlabel $\widetilde{v}_3$ at 187 134
\pinlabel $v_4$ at 199 124
\pinlabel $v_5$ at 217 127
\pinlabel $\widetilde{v}_6$ at 225 141
\pinlabel $v_7$ at 226 158
\pinlabel $v_8$ at 216 170
\pinlabel $\infty$ at 128 80
\pinlabel $S^{2}_{1}$ at 20 89
\pinlabel $S^{2}_{2}$ at 240 5
\pinlabel $C_u$ at 85 100
\pinlabel $C_v$ at 175 100
\pinlabel \rotatebox{-45}{$\underrightarrow{g}$} at 235 92
\endlabellist
\centering
\includegraphics{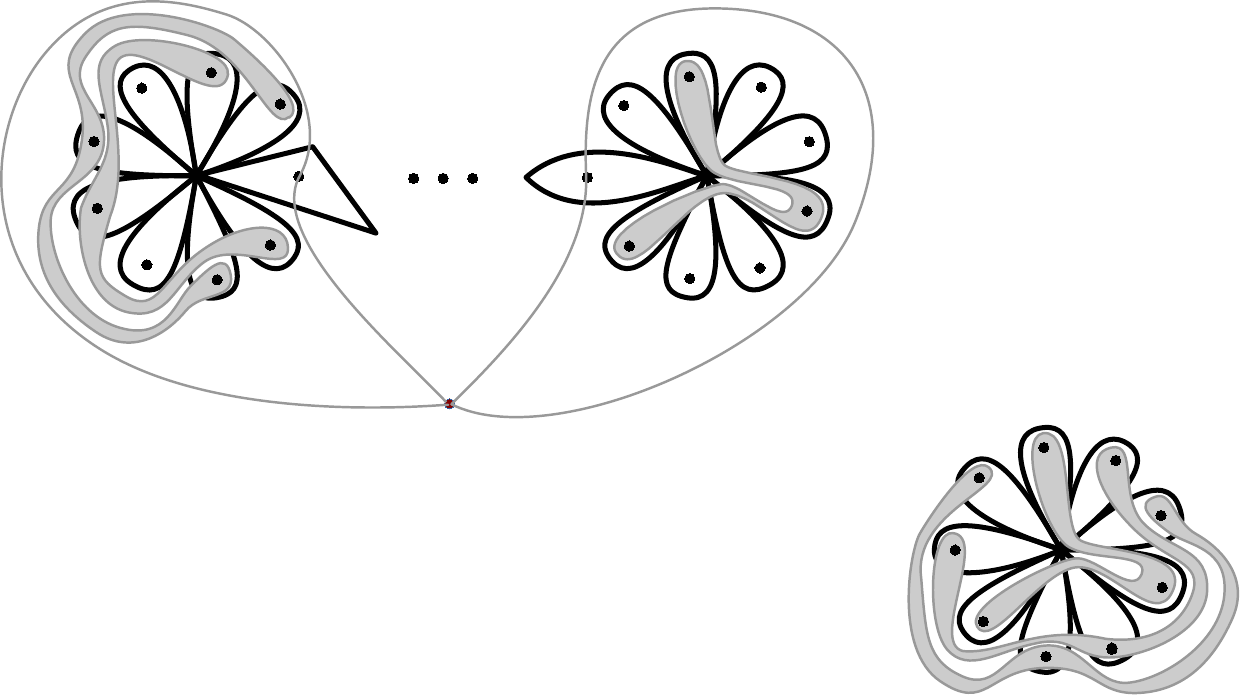}
\caption{Rose map $g:S_{1}^2\to S_{2}^2$}
\label{fig:rosep4}
\end{figure}

We next construct an orientation-preserving homeomorphism $h\co
S_2^2\to S_1^2$ such that (i) $h(v_j)=\widetilde{v}_{j+1}$ for $j\in
\{0,\dotsc,n-1\}$, (ii) $h(\infty )=g^{-1}(\infty )$ and (iii)
$h(D_i)=\widetilde{D}_{i+1}$ for $i\in \{0,\dotsc,m-1\}$, where
$D_m=D_0$.  See Figure \ref{fig:rosep6}.

\begin{figure}[!ht]
\labellist
\small\hair 2pt
\pinlabel $v_1$ at 26 60
\pinlabel $v_2$ at 22 45
\pinlabel $v_3$ at 9 25
\pinlabel $v_4$ at 38 6
\pinlabel $v_5$ at 61 9
\pinlabel $v_6$ at 66 36
\pinlabel $v_7$ at 67 51
\pinlabel $v_8$ at 62 81
\pinlabel $v_0$ at 41 85
\pinlabel $v_1$ at 236 60
\pinlabel $v_2$ at 232 45
\pinlabel $v_3$ at 219 25
\pinlabel $v_4$ at 248 6
\pinlabel $v_5$ at 271 9
\pinlabel $v_6$ at 276 36
\pinlabel $v_7$ at 277 51
\pinlabel $v_8$ at 272 81
\pinlabel $v_0$ at 251 85
\pinlabel $\widetilde{v_1}$ at 136 164
\pinlabel $\widetilde{v_2}$ at 118 174
\pinlabel $\widetilde{v_0}$ at 186 180
\pinlabel $\widetilde{v_4}$ at 65 161
\pinlabel $\widetilde{v_5}$ at 90 143
\pinlabel $\widetilde{v_6}$ at 239 142
\pinlabel $\widetilde{v_7}$ at 124 123
\pinlabel $\widetilde{v_3}$ at 186 126
\pinlabel $\widetilde{v_8}$ at 141 132
\pinlabel \rotatebox{315}{$\underrightarrow{g}$} at 222 106
\pinlabel \rotatebox{45}{$\underrightarrow{h}$} at 68 96
\pinlabel $\underrightarrow{f}$ at 156 40
\pinlabel $S^{2}_{1}$ at 146 104
\pinlabel $S^{2}_{2}$ at 1 0
\pinlabel $S^{2}_{2}$ at 306 0
\endlabellist
\centering
\includegraphics{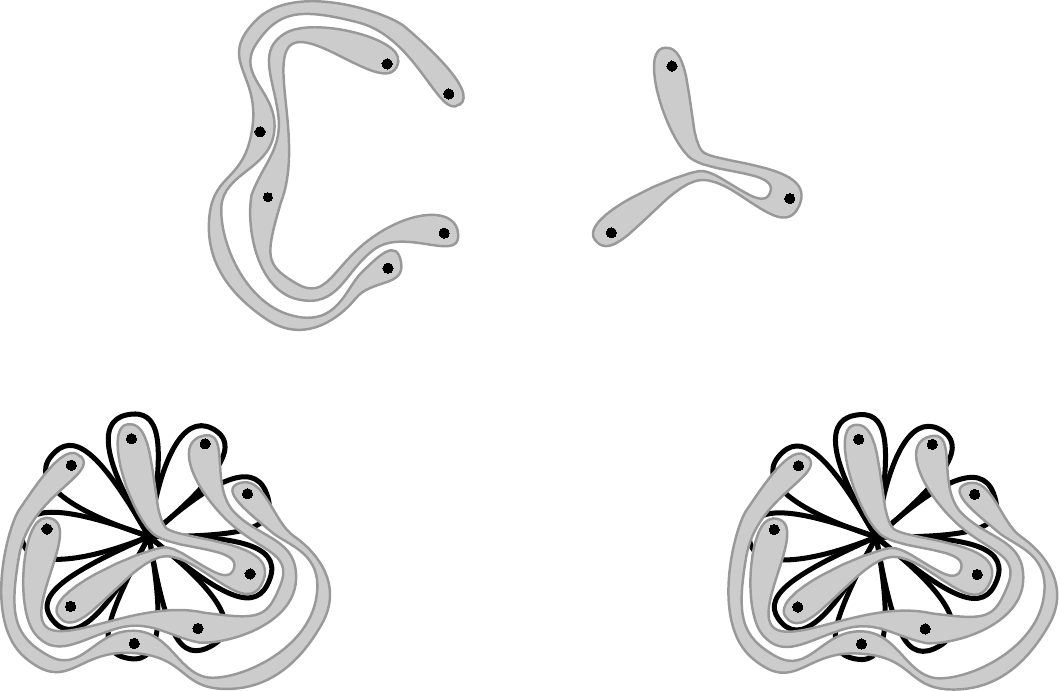}
\caption{The obstructed Thurston map $f:S^{2}_{2}\to S^{2}_{2}$}
\label{fig:rosep6}
\end{figure}

Finally, define $f:=g\circ h$.  The map $f$ is a Thurston map whose
portrait is isomorphic to $\Gamma$ and the boundaries of the disks
$D_i$ form a degenerate Levy cycle.  This proves Theorem
\ref{thm:obstructed} in cases (i) and (ii).

Now suppose that $\Gamma$ satisfies (iii).  Let $C$ be a non-attractor
cycle of length $\ell \ge 2$ that does not contain all of the finite
critical values of $\Gamma$.  Let $v_1,\dotsc,v_n$ be the finite
postcritical vertices of $\Gamma$ indexed so that $v_1,\dotsc,v_\ell $
are the vertices of $C$, $\tau(v_\ell)=v_1$ and $\tau(v_i)=v_{i+1}$
for $i\in\{1,\ldots,\ell-1\}$.  By Lemma \ref{lem:cmbllemma} there
exists a critical value $v\notin C$ and a rose map $g\co (S_1^2,
R_1)\to (S_2^2,R_2)$ realizing the branch data of $\Gamma$ such that
$n$ connected components of $S_2^2\setminus R_2$ each bounded by a
petal of $R_2$ are labeled $v_1,\dotsc,v_n$ in counterclockwise order,
the remaining connected component is labeled $\infty $ and the
boundary of one connected component $C_v$ of $S_1\setminus R_1^*$
contains exactly two critical points: one vertex with label $v$ and
one vertex with label $\infty$.  We may assume that the restriction of
$g$ to $C_v$ is injective.  As in cases (i) and (ii), we identify
every postcritical vertex $w$ of $\Gamma$ with a point in the
connected component of $S_2^2\setminus R_2$ with label $w$.

Let $D$ be a closed topological disk in the open disk $g(C_v)$ whose interior
contains $v_1,\dotsc,v_\ell $ but such that $D$ contains no other
postcritical vertex of $\Gamma$.  Let $\widetilde{D}$ be the lift of
$D$ to $C_v$.  Let $\widetilde{v}_i$ be the lift of $v_i$ to $C_v$ for
$i\in\{1,\ldots,\ell\}$.  Now construct an orientation-preserving
homeomorphism $h\co S_2^2\to S_1^2$ such that (i)
$h(v_i)=\widetilde{v}_{i+1}$ for $i\in\{1,\ldots,\ell-1\}$, (ii)
$h(v_\ell )=v_1$, (iii) $h(\infty )=g^{-1}(\infty )$, (iv)
$h(D)=\widetilde{D}$ and (v) the portrait of $f:=g\circ h$ is
isomorphic to $\Gamma$.  Since $\partial D$ is by itself a degenerate
Levy cycle for $f$, this proves Theorem \ref{thm:obstructed} in case
(iii).

Finally, we consider case (iv), where $\Gamma$ has at least two
non-attractor cycles of length one.  Let $a$ and $b$ denote the
vertices of two non-attractor cycles of $\Gamma$ of length $1$.
Define the abstract portrait $\Gamma'$ as follows:
$V(\Gamma')=V(\Gamma)$, $\tau'(a)=b$, $\tau'(b)=a$, and
$\tau'(x)=\tau(x)$ for all $x\in V(\Gamma')\setminus\{a,b\}$. So
$\Gamma'$ satisfies (iii).  The proof for case (iii) shows that there
is a Thurston map $f$ realizing $\Gamma'$ that has a degenerate Levy
cycle consisting of a single curve which bounds a disk $D$ such that
$D\cap P_f = \{a,b\}$.  Postcompose $f$ with an orientation-preserving
homeomorphism $h:S^{2}\to S^{2}$ such that $h$ is the identity in the
complement of $D$, $h(b)=a$ and $h(a)=b$. Then $h\circ f$ is an
obstructed Thurston map whose portrait is isomorphic to $\Gamma$.
\end{proof}

\section{Classification of completely unobstructed portraits}
\label{sec:classification}

\begin{proof}[Proof of Theorem \ref{thm:classification}] For the
forward direction, we will prove the contrapositive.  Suppose that
$\Gamma$ is an abstract polynomial portrait that does not satisfy any
of the conditions (i)-(iv).  Since $\Gamma$ doesn't satisfy (i) and
(ii), $\Gamma$ has at least four postcritical vertices and there is a
non-attractor cycle.  Since $\Gamma$ doesn't satisfy (iii), either
$\Gamma$ has at least two non-attractor cycles or there is a
non-attractor cycle of length at least two.  If $\Gamma$ has at least
two nonattractor cycles, then it satisfies conditions (iii) or (iv) of
Theorem \ref{thm:obstructed} and so $\Gamma$ can be realized by an
obstructed Thurston map.  Now suppose that $\Gamma$ has a single
non-attractor cycle and it has length at least two.  If this cycle
doesn't contain all of the finite postcritical vertices, then $\Gamma$
can be realized by an obstructed map by condition (iii) of Theorem
\ref{thm:obstructed}.  So we may suppose that all of the finite
postcritical vertices are in a single non-attractor cycle.  If its
length is not a prime power, then $\Gamma$ can be realized by an
obstructed map by condtion (ii) of Theorem \ref{thm:obstructed}.  If
its length is a prime power, then since $\Gamma$ doesn't satisfy (iv)
it satisfies (i) of Theorem \ref{thm:obstructed} and $\Gamma$ can be
realized by an obstructed map.

For the reverse direction, suppose $\Gamma$ is an abstract polynomial
portrait and $f$ is a Thurston map whose portrait is isomorphic to
$\Gamma$.  If $\Gamma$ satisfies (i), then $f$ is unobstructed by
Thurston's characterization theorem since there aren't enough
postcritical points to have an obstruction.  If $\Gamma$ satifies
(ii), then $f$ is unobstructed by the Levy-Berstein theorem.  If
$\Gamma$ satisfies (iii) or (iv), then $f$ is unobstructed by Theorem
\ref{thm:unobstructed}.
\end{proof}


\end{document}